\documentclass[11pt]{amsart}
\usepackage[divide={2.45cm,*,2.45cm}]{geometry} 
\usepackage{comment}
\usepackage{graphicx}
\usepackage{amssymb}
\usepackage{epstopdf}

\excludecomment{versiona}
\excludecomment{versionb}
\title{Stability and Convergence of the Sasaki-Ricci Flow}
\newcommand {\pl}[1]{\ensuremath{\frac{\partial}{\partial#1}}}
\newcommand{\dbar}{\ensuremath{\bar{\partial}}}
\newcommand{\dop}[1]{\ensuremath{\partial_{#1}}}
\newcommand{\dbop}[1]{\ensuremath{\partial_{\bar{#1}}}}
\newcommand{\RmT}[4]{\ensuremath{Rm^{T}_{\bar{#1}#2}{}^{#3}{}_{#4}}}
\newcommand{\Rm}[4]{\ensuremath{Rm_{{#1}{#2}}{}^{#3}{}_{#4}}}
\newcommand{\E}{\ensuremath{\mathcal{E}}}

\author{Tristan C. Collins}
\address{Department of Mathematics, Columbia University, New York, NY 10027}
\email{tcollins@math.columbia.edu}

\begin{document}
\theoremstyle{plain}
\newtheorem{Lemma}{Lemma}[section]
\newtheorem{prop}{Proposition}[section]
\newtheorem{Theorem}{Theorem}[section]
\newtheorem{corollary}{Corollary}[section]
\newtheorem{definition}{Definition}[section]
\newtheorem{example}{Example}
\theoremstyle{remark}
\newtheorem*{remark}{Remark}
\newtheorem{claim}{Claim}
\maketitle

\begin{abstract}
We introduce a holomorphic sheaf $\E$ on a Sasaki manifold $S$ and study two new notions of stability for $\E$ along the Sasaki-Ricci flow related to the `jumping up' of the number of global holomorphic sections of $\E$ at infinity.  First, we show that if the Mabuchi K-energy is bounded below, the transverse Riemann tensor is bounded in $C^{0}$ along the flow, and the  $C^{\infty}$ closure of the Sasaki structure on $S$ under the diffeomorphism group does not contain a Sasaki structure with strictly more global holomorphic sections of $\E$, then the Sasaki-Ricci flow converges exponentially fast to a Sasaki-Einstein metric.  Secondly, we show that if the Futaki invariant vanishes, and the lowest positive eigenvalue of the $\dbar$ Laplacian on global sections of $\E$ is bounded away from zero uniformly along the flow, then the Sasaki-Ricci flow converges exponentially fast to a Sasaki-Einstein metric.
\end{abstract}

\section{Introduction}
\begin{versionb}
There has recently been a great deal of interest in Sasaki geometry, particularly with regards to producing examples of Sasaki-Einstein, and transversely K\"ahler-Einstein metrics.  This interest has been fueled primarily by advances in the AdS/CFT correspondence in theoretical physics.  When the basic first Chern class is non-positive, the existence theory for canonical Sasaki metrics is well developed and generalizes the results in the K\"ahler setting.  However, when the basic first Chern class is positive there are known obstructions to existence; see for example the work of Futaki, Ono, and Wang \cite{Futaki}, and Gauntlett, Martelli, Sparks and Yau \cite{GMarSparYau}.  It is reasonable to expect that a suitably generalized version of the now famous conjecture of Yau \cite{Yau} should hold in the Sasaki setting;  that is, existence of Sasaki-Einstein metrics should be equivalent to some geometric invariant theory notion of stability.  We refer the reader to the survey article \cite{PSstab} and the references therein for an introduction to this very active area of research in the K\"ahler setting.  

A great deal of ingenuity on the part of many different authors has produced a large number of quasi-regular Sasaki-Einstein manifolds; we refer the reader to the survey article \cite{Sparks} and the references therein.  Until recently there were no examples of Sasaki-Einstein metrics on irregular Sasaki manifolds;  in fact, such manifolds were conjectured to not exist by Cheeger and Tian \cite{CheeTian}.  The first examples of irregular Sasaki-Einstein manifolds were produced by Gauntlett, Martelli, Sparks and Waldram \cite{GMarSparW, GMarSparW1}, and further developed by Martelli and Sparks \cite{MarSpar, MarSpar1}.  New toric irregular Sasaki-Einstein metrics have been produced, for example on the circle bundle of a power of the anti-canonical bundle over the two point blow up of $\mathbb{C}P^{2}$ \cite{Futaki}.  These discoveries necessitate a more robust approach to the  existence theory for Sasaki-Einstein structures.  In response, a number of approaches have been developed by several different authors.  For example, Martelli, Sparks and Yau \cite{MarSparYau}, and Boyer, Galicki and Simanca \cite{BoyGalSim} developed an approach in the spirit of Calabi, by examining functionals related to the volume and scalar curvature.  A flow approach was developed by Smoczyk, Wang and Zhang \cite{SmoWaZa}, where they introduced the Sasaki-Ricci flow, and generalized the results of Cao \cite{Cao} to Sasaki manifolds.  

In this paper, we further develop the theory of stability for Sasaki manifolds.  Our point of view shall be the flow approach.  The type of result we are interested in was first obtained in the K\"ahler case by Phong and Sturm \cite{PS}; they showed that a form of stability for the complex structure implies the convergence of the K\"ahler-Ricci flow, in the presence of uniform bounds on the Riemann tensor. Phong and Sturm considered the following forms of stability on a K\"ahler manifold $(X,J,\omega)$ with positive first Chern class;
\begin{enumerate}
\item[(A)] The Mabuchi energy is bounded from below
\item[(B)] Let $J$ be the complex structure of $X$, viewed as a tensor.  Then the $C^{\infty}$ closure of the orbit of J under the diffeomorphism group of $X$ does not contain any complex structure $J_{\infty}$ with the property that the space of holomorphic vectorfields with respect to $J_{\infty}$ has dimension strictly higher than the dimension of the space of holomorphic vector fields with respect to $J$
\end{enumerate}

Condition (B) is a geometric manifestation of stability in the sense that when it does not hold, the moduli space of complex structures cannot be Hausdorff in the topology of $C^{\infty}$ convergence.  Assuming a uniform bound on the Riemann tensor along the K\"ahler-Ricci flow is, of course, very restrictive.  In \cite{PSSW}, Phong, Song, Sturm and Weinkove managed to remove the assumption on the boundedness of the Riemann tensor while imposing a more general stability condition, which they referred to as condition (S). !!!!!!!!!!!
\end{versionb}
\begin{versionb}
 (cf. Theorem 1.2 below).  Combining the results of \cite{PS, PSSW}, we have;
\begin{Theorem}[\cite{PS} Theorem 1, \cite{PSSW} Theorem 1]
Let $(X,J)$ be a compact complex manifold of dimension $n$.  Let $\dot{g}_{\bar{k}j} = -R_{\bar{k}j}+\kappa g_{\bar{k}j}$ be the normalized K\"ahler-Ricci flow, with initial metric $h_{\bar{k}j}(0)$, and $\kappa g_{\bar{k}j}(0)$ a K\"ahler metric in the first Chern class of X.  Here, $\kappa n$ denotes the total scalar curvature.  Assume that the Riemann curvature tensor is uniformly bounded along the flow.
\begin{enumerate}
\item[{\it (i)}] If condition (A) holds, then we have
\begin{enumerate}
\item[(a)] for any $s\geq 0$
\begin{equation*}
\lim_{t\rightarrow \infty}\|R_{\bar{k}j}(t) - \kappa g_{\bar{k}j}(t)\|_{(s)} =0
\end{equation*}
where $\| \cdot \|_{(s)}$ denotes the Sobolev norm of order $s$ with respect to the metric $g_{\bar{k}j}(t)$.
\item[(b)]Moreover,
\begin{equation*}
 \int_{0}^{\infty}  \|R(t)-\kappa n\|^{p}_{C^{0}}dt <\infty \quad \text{  for any  } p >2.
 \end{equation*}
\item[(c)]In particular, 
\begin{equation*}
 \|R(t)-\kappa n\|_{C^{0}} \rightarrow 0 \quad \text{ as } t \rightarrow \infty.
 \end{equation*}
\end{enumerate}
\item[{\it (ii)}] If both conditions (A) and (B) hold, then the K\"ahler-Ricci flow converges exponentially fast in $C^{\infty}$ to a K\"ahler-Einstein metric.
\end{enumerate}
\end{Theorem}

\begin{Theorem}[\cite{PSSW} Theorem 2]
Fix $\omega_{0} \in \pi c_{1}(X)$.  Let $g_{\bar{k}j}(t)$ be the solution of the K\"ahler-Ricci flow with initial value $(g_{0})_{\bar{k}j} \in \pi c_{1}(X)$.  Let $\lambda_{t}$ be the lowest strictly positive eigenvalue of the Laplacian $\dbar^{\dagger}\dop{} = -g^{j\bar{k}}\nabla_{j}\nabla_{\bar{k}}$ acting on smooth $T^{1,0}$ vectorfields.
\begin{enumerate}
\item[{\it (i)}] If the condition (A) is satisfied and condition
\begin{enumerate}
\item[(S)] $\inf_{t\in[0,\infty)} \lambda_{t} >0$
\end{enumerate}
holds, then the metrics $g_{\bar{k}j}(t)$ converge exponentially fast in $C^{\infty}$ to a K\"ahler-Einstein metric.
\item[{\it (ii)}]
Conversely, if the metrics $g_{\bar{k}j}(t)$ converge in $C^{\infty}$ to a K\"ahler-Einstein metric, then conditions (A) and (S) are satisfied.
\item[{\it (iii)}]In particular, if the metrics $g_{\bar{k}j}(t)$ converge in $C^{\infty}$ to a K\"ahler-Einstein metric, then they converge exponentially fast in $C^{\infty}$ to this metric.
\end{enumerate}
\end{Theorem}
\end{versionb}

Sasaki geometry is a generalization of K\"ahler geometry with applications to the AdS/CFT correspondence in theoretical phyiscs.  It is an important problem to determine when Sasaki-Einstein metrics exist.  When the basic first Chern class is non-positive, the existence theory is well developed and generalizes the results of Aubin and Yau \cite{Aubin, Yau2}.  However, when the basic first Chern class is positive, there are known obstructions to the existence of Sasaki-Einstein metrics; see, for example, the work of Futaki, Ono, and Wang \cite{Futaki}, and Gauntlett, Martelli, Sparks and Yau \cite{GMarSparYau}.  It is expected that a suitably modified version of the famous conjecture of Yau \cite{Yau} should hold in the Sasaki setting; that is, existence of Sasaki-Einstein metrics with positive basic first Chern class should be equivalent to some geometric invariant theory notion of stability.  Recently, a flow approach to the Sasaki-Einstein problem was developed by Smoczyk, Wang and Zhang in \cite{SmoWaZa}, where they introduced the Sasaki-Ricci flow which generalizes the K\"ahler-Ricci flow, and extended the results of Cao \cite{Cao}.  There is now a large body of work relating various notions of algebraic stability to convergence of the K\"ahler-Ricci flow; see for example \cite{PSstab, Gabor, Valent} and the references therein. It is desirable to determine whether analogues of these results hold in the Sasaki case.  A particular form of stability which arises in the K\"ahler setting concerns the degeneration of eigenvalues of various Laplacians along the flow.  Phong and Sturm \cite{PS}, and Phong, Song, Sturm and Weinkove \cite{PSSW} proved convergence of the K\"ahler-Ricci flow assuming a bound below for the Mabuchi functional and stability conditions for the lowest positive eigenvalue of the $\dbar$ Laplacian on $T^{1,0}$ vector fields (cf. condition(B) in \cite {PS}, and condition (S) in \cite{PSSW}).  In \cite{Zhang}, Zhang proved convergence of the flow under a non-degeneracy condition for the `second' eigenvalue of a modified Laplacian on smooth functions.  

In the Sasaki case, it is natural to ask whether forms of stability analogous to those studied in \cite{PS, PSSW, Zhang} are available, and whether they imply the convergence of the Sasaki-Ricci flow to a transverse K\"ahler-Einstein metric when the Futaki invariant vanishes or the Mabuchi functional is bounded below.  We aim to address these questions presently.  We would like to point out a few simple observations which hint at the difficulties ahead.  Assume for simplicity that the Sasaki structure is regular, so that the Sasaki manifold $S$ is diffeomorphic to a K\"ahler manifold $S/U(1)$ with a principle $U(1)$ bundle.  In this case, the stability conditions we seek are necessarily the pull back of the K\"ahler stability conditions under the quotient map $\pi: S\rightarrow S/U(1)$.  The first observation is that pulling back sections of the tangent bundle by $\pi$ does not yield a module over the ring of smooth functions.  Thus, if we seek to generalize condition (B) of \cite{PS} or condition (S) of \cite{PSSW} we must work in the realm of locally free sheaves of modules over the ring of \emph{basic} functions.  More generally, when the Sasaki structure is irregular, so that the leaf space of the Reeb foliation does not have the structure of a K\"ahler orbifold, how do we identify the space of ``holomorphic vector fields"?  A general approach to this problem is to extend K\"ahler notions of stability to K\"ahler orbifolds and then formulate some related notion of stability which behaves well under approximation by quasi-regular Sasaki structures.  We prefer the point of view which avoids these approximation techniques.  In \cite{BoyGalSim, Futaki, NS}, the Lie algebra of holomorphic Hamiltonian vector fields was identified as a central object of study in the existence of Sasaki-Einstein and extremal Sasaki metrics.  Is it possible to view these vector fields as the kernel of a $\dbar$ operator on the global sections of some sheaf $\E$?  More importantly, if such an $\E$ exists, can we relate the convergence of the Sasaki-Ricci flow to the eigenvalues of the $\dbar$ Laplacian on the global sections of $\E$?  In this paper we answer these questions in the affirmative.  We identify a sheaf $\E$, called the sheaf of transverse foliate vector fields, which has a well defined $\dbar$ operator, with the property that the global holomorphic sections of $\E$ correspond precisely to the Hamiltonian holomorphic vector fields.  We consider the following notions of stability;
\begin{enumerate}
\item[(M)] The Mabuchi energy is bounded below.
\item[(F)] The Futaki invariant vanishes.
\item[(C)] 
Let $(\xi, \eta, \Phi)$ be a Sasaki structure on S. Then the $C^{\infty}$ closure of the orbit of the triple $(\xi, \eta, \Phi)$ under the diffeomorphism group of $S$ does not contain any Sasaki structure $(\xi_{\infty}, \eta_{\infty}, \Phi_{\infty})$ with the property that the dimension of the space of global holomorphic sections of the sheaf of transverse foliate vector fields with respect to  $(\xi_{\infty}, \eta_{\infty}, \Phi_{\infty})$ has dimension strictly higher than the dimension of the space of global holomorphic sections of the sheaf of transverse foliate vector fields with respect to $(\xi, \eta, \Phi)$.  
\end{enumerate}
Condition (C) generalizes condition (B) of \cite{PS}.  In light of Proposition~\ref{vanishing fut prop} below, condition (F) is at least \emph{a priori} weaker than condition (M).  We refer the reader to \S5 for details on the sheaf $\E$, and its holomorphic structure.  Our first theorem extends Theorem 1 in \cite{PS}.
\begin{Theorem}\label{main theorem}
Let $(S, \xi, \eta, \Phi, g_{0})$ be a compact Sasaki manifold with $c_{1}^{B}(S) >0$.  Assume that $g(t)$ is a solution of the Sasaki-Ricci flow with $g(0) = g_{0}$, and $(2n+2) g_{0}$ is in the basic first Chern class of $(S,\xi,\eta,\Phi, g_{0})$.  Assume that the transverse Riemann curvature is bounded along the flow.
\begin{enumerate}
\item[\it{(i)}] If condition (M) holds, then we have for any $s\geq 0$
\begin{equation*}
\lim_{t\rightarrow \infty} \|R^{T}(t) -(2n+2) g^{T}(t)\|_{(s)} =0
\end{equation*}
where $\| \cdot \|_{(s)}$ denotes the Sobolev norm of order s with respect to the metric $g(t)$.
\item[{\it(ii)}] If both conditions (M) and (C) hold, then the Sasaki-Ricci flow converges exponentially fast in $C^{\infty}$ to a Sasaki-Einstein metric.
\end{enumerate}
\end{Theorem}

We remove the condition on the boundedness of $Rm^{T}$ by introducing the stability condition (T), which generalizes condition (S) of \cite{PSSW}.  The following theorem extends the results of \cite{PSSW} , and \cite{Zhang}.

\begin{Theorem}\label{main theorem 2}
Let $(S, \xi, \eta, \Phi, g_{0})$ be a compact Sasaki manifold with $c_{1}^{B}(S) >0$.  Assume that $g(t)$ is a solution of the Sasaki-Ricci flow with $g(0) = g_{0}$, and $(2n+2) g_{0} \in c_{1}^{B}(S)$.   Let $\lambda_{t}$ be the lowest strictly positive eigenvalue of the Laplacian $\square_{\E} := -(g^{T})^{j\bar{k}}\nabla^{T}_{j}\nabla^{T}_{\bar{k}}$ acting on smooth global sections of $\E^{1,0}$.
\begin{enumerate}
\item[{\it (i)}] If condition (F) and condition\\
\begin{enumerate}
\item[(T)] $\inf_{t\in[0,\infty)} \lambda_{t} >0$\\
\end{enumerate}
hold, then the metrics $g(t)$ converge exponentially fast in $C^{\infty}$ to a Sasaki-Einstein metric.
\item[{\it (ii)}]
Conversely, if the metrics $g(t)$ converge in $C^{\infty}$ to a Sasaki-Einstein metric, then conditions (F) and (T) hold.
\item[{\it (iii)}]In particular, if the metrics $g(t)$ converge in $C^{\infty}$ to a Sasaki-Einstein metric, then they converge exponentially fast in $C^{\infty}$ to this metric.
\end{enumerate}
\end{Theorem}
\begin{versionb}
Combining the techniques in the proofs of Theorems~\ref{main theorem}, and~\ref{main theorem 2} we obtain
\begin{Theorem}\label{main theorem 3}
Let $(S, \xi, \eta, \Phi, g_{0})$ be a compact Sasaki manifold with $c_{1}^{B}(S) >0$.  Assume that $g(t)$ is a solution of the Sasaki-Ricci flow with $g(0) = g_{0}$, and $(2n+2) g_{0}$ is in the basic first Chern class of $(S,\xi,\eta,\Phi)$.  If condition (A) holds, then we have
\begin{enumerate}
\item[\it{(i)}]
 \begin{equation*}
 \int_{0}^{\infty}  \|R^{T}(t)-(2n+2) n\|^{p}_{C^{0}}dt <\infty \quad \text{  for any  } p >2.
 \end{equation*}
\item[{\it (ii)}] In particular, 
\begin{equation*}
 \|R^{T}(t)-(2n+2) n\|_{C^{0}} \rightarrow 0 \quad \text{ as } t \rightarrow \infty.
 \end{equation*}
\end{enumerate}
\end{Theorem}
\end{versionb}

The condition ``$(2n+2) g_{0}$ is in the basic first Chern class of $(S,\xi,\eta,\Phi, g_{0})$" in the above Theorems is not restrictive in light of the so-called ``$\mathcal{D}$- homothetic transformations" introduced by Tanno \cite{Tanno};  see \S3.  The outline of this paper is as follows;  in \S2 we provide an introduction to Sasaki geometry.  In \S3 we discuss perturbations of Sasaki structures and the Sasaki-Ricci flow.  We present an argument for a specific choice of the initial value of the transverse K\"ahler potential and point out a some consequences of this normalization.  We also present previous results on the Sasaki-Ricci flow which we need.  In \S4 we extend the well known estimates of Hamilton \cite{RHam2} and Shi \cite{Shi} for the Ricci flow to the Sasaki-Ricci flow on compact Sasaki manifolds.  In particular, we prove

\begin{Theorem}\label{BBS Thm}
Let $(S,g_{0})$ be a compact Sasaki manifold of dimension $2n+1$, and suppose that $g(t)$ is a solution of the normalized Sasaki-Ricci flow, with $g(0)=g_{0}$.  Then, for each $\alpha >0$, and every $m \in \mathbb{N}$, there exists a constant $C_{m}$ depending only on $m, n$ and $\max \{\alpha, 1\}$ such that if $K$  satisfies
\begin{equation*}
|Rm^{T}(x,t)|_{g^{T}(x,t)} \leq K  \text{ for every } x \in S, \text{ and } t\in [0, \frac{\alpha}{K}]
\end{equation*}
then, the bound
\begin{equation*}
\max \left\{|\nabla ^{m} Rm(x,t)|_{g(t)}, \text{ }|\nabla^{m}Rm^{T}(x,t)|_{g^{T}(x,t)} \right\}\leq \frac{C_{m}\max\{K^{1/2},K\}}{t^{m/2}}
\end{equation*}
holds for every $x \in S$ and $ t \in (0, \frac{\alpha}{K}]$.
\end{Theorem}

In \S5  we take up stability on Sasaki manifolds.  We begin by discussing the Futaki invariant and the Mabuchi energy, extending some well known results from the K\"ahler theory to the Sasaki setting.  We then construct the sheaf $\E$ and discuss its properties.  In \S6 we prove Theorem~\ref{main theorem} part {\it(i)}, and in \S7 we reduce the proof of Theorem~\ref{main theorem} part {\it(ii)} to obtaining a positive lower bound for the smallest positive eigenvalue of the $\dbar$ Laplacian on the global sections of $\E^{1,0}$.  In \S8 we complete the proof of Theorem~\ref{main theorem} by discussing some compactness results for Sasaki manifolds, and showing that condition (C) implies a positive lower bound for the smallest positive eigenvalue of the $\dbar$ Laplacian.  In \S9 we prove Theorem~\ref{main theorem 2}. 
\\

{\bf Acknowledgements:}
I would like to thank my advisor Professor D.H. Phong for his guidance and encouragement, as well as for suggesting this problem.  I would also like to thank Professor Valentino Tosatti for many helpful conversations. 

\section{Sasaki Manifolds and Transverse K\"ahler Geometry}
In this section we give a brief introduction to Sasaki geometry.  For a more thorough introduction we refer the reader to \cite{BoyGal1, TristC, Sparks}.  A Sasakian manifold of dimension $2n+1$ is a Riemannian manifold $(S^{2n+1},g)$ with the property that its metric cone $(C(S) = \mathbb{R}_{>0}\times S, \overline{g} = dr^2 + r^{2}g)$ is K\"ahler.  A great deal of the geometry of Sasaki manifolds is induced by the \emph{Euler} vector field $r\dop{r}$.  One can easily show that $r\dop{r}$ is real holomorphic, that is, $\mathcal{L}_{r\dop{r}}J=0$.  A particularly important role in Sasaki geometry is played by the Reeb vector field, which is naturally induced from the vector field $r\dop{r}$.  
\begin{definition}
The Reeb vector field is $\xi = J(r\partial_{r})$, where $J$ denotes the integrable complex structure on $C(S)$.  
\end{definition}
Again, the Reeb field $\xi$ satisfies $\mathcal{L}_{\xi}J=0$, and so $\xi$ is real holomorphic.  The restriction of $\xi$ to the slice $\{r=1\}$  is a unit length Killing vector field, and its orbits thus define a one-dimensional foliation of S by geodesics called the Reeb foliation.  
\begin{versionb}
There is a dual one-form $\eta$ defined by $\eta = id^{c}\log r = i (\overline{\partial} -\partial)\log r$, which has the properties
\begin{equation*}
\eta(\xi)=1, \quad i_{\xi}d\eta=0.
\end{equation*}
Moreover, we have that for all vector fields X
\begin{equation*}
\eta(X) = \frac{1}{r^{2}}\overline{g}(\xi,X).
\end{equation*}
The K\"ahler form on $C(S)$ is given by $\omega = \frac{1}{2}d(r^{2}\eta)$.  The 1-form $\eta$ restricts to a 1-form $\eta |_{S}$ on $S \subset C(S)$.  Using the fact the $\mathcal{L}_{r\partial_{r}}\eta =0$ one shows that in fact $\eta = p^{*}(\eta |_{S})$ where $p$ is the canonical projection of $C(S)$ to $S$.  By abuse of notation, we do not distinguish between $\eta$ and its restriction to $S$.  Since the K\"{a}hler 2-form $\omega$ is non-vanishing, it follows that the top degree form $\eta \wedge (d\eta)^{n}$ is non-vanishing on $S$, and hence $(S,\eta)$ is a contact manifold.
\end{versionb}
Let $L_{\xi}$ be the line bundle spanned by the non-vanishing vector field $\xi$.  The contact subbundle $D\subset TS$ is defined as $D= \ker \eta$ where $\eta(X) := g(\xi, X)$.  We have the exact sequence
\begin{equation}\label{exact seq}
0\rightarrow L_{\xi} \rightarrow TS \xrightarrow{p} Q \rightarrow 0.
\end{equation}
The Sasakian metric $g$ gives an orthogonal splitting of this sequence $\sigma : Q \rightarrow D$ so that we identify $Q \cong D$, and $TS = D\oplus L_{\xi}$. Define a section $\Phi \in End(TS)$ via the equation $\Phi(X) = \nabla_{X} \xi$.  One can then check that $\Phi |_{D} = J |_{D}$ and $\Phi |_{L_{\xi}} = 0$, and that
\begin{equation*}
\Phi^{2} = -1 + \eta \otimes \xi \quad \text{, and   } \quad g(\Phi(X), \Phi(Y)) = g(X,Y)-\eta(X)\eta(Y)
\end{equation*}
for any vector fields $X$ and $Y$ on S.  In particular, $g |_{D}$ is a Hermitian metric on D. 
\begin{versionb}
 Since
\begin{equation}\label{metric}
g(X,Y) = \frac{1}{2} d\eta (X, \Phi (Y)) + \eta(X)\eta(Y),
\end{equation}
 we see that $\frac{1}{2} d\eta |_{D}$ is the fundamental 2-form associated to $g |_{D}$.  
 \end{versionb}
 The triple $(D, \Phi |_{D}, d\eta)$ gives $S$ a transverse K\"{a}hler structure. FOR MORE SEE REF?!
\begin{versionb}
\subsection{Transverse K\"{a}hler structures and the Reeb foliation}

The Reeb foliation inherits a transverse holomorphic structure and K\"{a}hler metric in the following (explicit) way.  The leaf space of the foliation induced by the Reeb field can clearly be identified with the leaf space of the \emph{holomorphic} vector field $\xi -iJ(\xi)$ on the cone $C(S)$.  Here, $J$ denotes the integrable complex structure on the cone.  By using holomorphic, foliated coordinates on $C(S)$, we may introduce a foliation chart  $\{U_{\alpha}\}$ on $S$, where each $U_{\alpha}$ is of the form $U_{\alpha} = I \times V_{\alpha}$ with $I\subset \mathbb{R}$ an open interval, and $V_{\alpha} \subset \mathbb{C}^{n}$.  We can find coordinates $(x,z^{1}, \dots, z^{n})$ on $U_{\alpha}$, where $\xi=\partial_{x}$, and $z^{1},\dots,z^{n}$ are complex coordinates of $V_{\alpha}$.  We denote by $\pi_{\alpha}$, the map defined by
\begin{equation}\label{pi def}
\pi_{\alpha} : U_{\alpha} \rightarrow V_{\alpha}\subset \mathbb{C}^{n}
\end{equation}
The fact that the cone is complex implies that the transition functions between the $V_{\alpha}$ are holomorphic.  More precisely, if $(y, w^{1}, \dots, w^{n})$ are similarly defined  coordinates on $U_{\beta}$ with $U_{\alpha} \cap U_{\beta} \neq \emptyset$, then
\begin{equation*}
\frac{\partial z^{i}}{\partial \overline{w}^{j}} = 0 \text{ ,  } \quad \frac{\partial z^{i}}{\partial y} = 0,
\end{equation*}
and so $\pi_{\alpha} \circ \pi_{\beta}^{-1}$ is a biholomorphism on its domain.  Recall that the subbundle $D$ is equipped with the almost complex structure $J |_{D}$, so that on $D\otimes \mathbb{C}$ we may define the $\pm i$ eigenspaces of $J |_{D}$ as the $(1,0)$ and $(0,1)$ vectors respectively.  Then, in the above foliation chart, $(D \otimes \mathbb{C})^{(1,0)}$ is spanned by $\partial_{z^{i}} -\eta(\partial_{z^{i}})\xi$.
Since $\xi$ is a killing vector field it follows that $g|_{D}$ gives a well-defined Hermitian metric $g_{\alpha}^{T}$ on the patch $V_{\alpha}$.  Moreover, ~(\ref{metric}) implies that
\begin{equation*}
d\eta(\partial_{z^{i}} -\eta(\partial_{z^{i}})\xi, \partial_{\overline{z}^{j}} -\eta(\partial_{\overline{z}^{j}})\xi)=d\eta(\partial_{z^{i}},\partial_{\overline{z}^{j}}).
\end{equation*}
Thus, the fundamental 2-form $\omega_{\alpha}^{T}$ for the Hermitian metric $g_{\alpha}^{T}$ in the patch $V_{\alpha}$ is obtained by restricting $\frac{1}{2}d\eta$ to a fibre $\{x=constant\}$.  It follows that $\omega_{\alpha}^{T}$ is closed, and the transverse metric $g_{\alpha}^{T}$ is K\"{a}hler.  In the following subsection we will introduce coordinates which make plain the transverse K\"ahler structure we are now describing.  We delay these developments so that we may first introduce the notions of transverse geometry which will be important for us.  

Sasaki manifolds fall in to three categories based on the orbits of the Reeb field.  If the orbits of the Reeb field are all closed, then the $\xi$ generates a locally free, isometric $U(1)$ action on $(S,g)$.  If the $U(1)$ action is free, then $(S,g)$ is said to be \emph{regular} and the quotient manifold $S/U(1)$ is K\"ahler.  If the action is not free, then $(S,g)$ is said to be \emph{quasi-regular}, and the quotient manifold is a K\"ahler orbifold.  If the orbits of $\xi$ are not all closed, then the Sasakian manifold $(S,g)$ is said to be irregular. 
\end{versionb}

The identification $Q \cong D$ endows the quotient bundle $Q$ with a transverse metric $g^{T}$.  There is a unique, torsion-free connection on $Q$,  which is compatible with the metric $g^{T}$.  This connection is defined by
\begin{displaymath}
\nabla^{T}_{X}V = \left\{ \begin{array}{lr}
	\left(\nabla_{X}\sigma(V)\right)^{p},& \text{if } X \text{ is a section of D}\\
	\left[\xi, \sigma(V)\right]^{p}, &\text{ if } X=\xi
\end{array}
\right.
\end{displaymath}
where $\nabla$ is the Levi-Civita connection on $(S,g)$, $\sigma$ is the splitting map induced by $g$, and $p: TS \rightarrow Q$ is the projection.  This connection is called the \emph{transverse Levi-Civita connection} as it is compatible with $g^{T}$, and torsion free with respect to the bracket induced on $Q$.
\begin{versionb}
as it satisfies
\begin{equation*}
\nabla_{X}^{T}Y - \nabla^{T}_{Y}X - [X,Y]^{p}=0,
\end{equation*}
\begin{equation*}
Xg^{T}(V,W) = g^{T}(\nabla_{X}^{T}V,W)+ g^{T}(V, \nabla^{T}_{X}W).
\end{equation*}
In this way it is easy to see that the transverse Levi-Civita connection is the pullback of the Levi-Civita connection on the local Riemannian quotient.  
\end{versionb}
We will denote by $Rm^{T}$ the curvature operator defined by the connection $\nabla^{T}$, $Ric^{T}$, and $R^{T}$ will denote the transverse Ricci curvature of scalar curvature respectively.
The geometry of the manifold $S$ is largely controlled by its transverse geometry.  For local sections $X,Y,Z,W$ of $D$, the curvature of $(S,g)$ is related to the curvature of $(Q,g^{T})$ by 
\begin{equation}\label{curvature relation}
\begin{aligned}
Rm(X,Y,Z,W) = &Rm^{T}(X,Y,Z,W)+ g(\Phi(X),Z)g(\Phi(Y),W)\\
&-g(\Phi(X),W)g(\Phi(Y),Z) +2g(\Phi(X),Y)g(\nabla_{\xi}X, W).
\end{aligned}
\end{equation}
The geometry orthogonal to the distribution $D$ is uniform in the sense that for any vector fields $X,Y \in TS$ there holds
\begin{equation}\label{curvature relation 2}
Rm(X,Y)\xi = \eta(Y)X-\eta(X)Y,
\end{equation}
\begin{equation}\label{curvature relation 3}
Rm(X,\xi)Y = \eta(Y)X-g(X,Y)\xi.
\end{equation}
We refer the reader to \cite{BoyGal1, Futaki} for more on these standard formulae.

\begin{definition}
For $x \in S$, let $orb_{\xi}x$ denote the orbit of $x$ under the action generated by the Reed field.  We define the transverse distance function $d^{T}: S\times S \rightarrow \mathbb{R}$ by
\begin{equation*}
d^{T}(x,y) = \inf_{\{p\in orb_{\xi}x, q \in orb_{\xi}y\}} dist(p,q)
\end{equation*}
\end{definition}

The transverse distance function played a crucial role in \cite{TristC}.  With these preparations, we see it fitting to make a brief digression on coordinates.  One can choose local coordinates $(x, z^{1},\dots,z^{n})$ on a small neighbourhood $U$ of $p$ such that
\begin{enumerate}
\item[\textbullet] $\xi = \frac{\partial}{\partial x}$
\item[\textbullet] $ \eta = dx + \sqrt{-1}\sum_{j=1}^{n}h_{j}dz^{j} - \sqrt{-1}\sum_{j=1}^{n}h_{\bar{j}}d\bar{z}^{j}$
\item[\textbullet] $ \Phi = \sqrt{-1} \left\{ \sum_{j=1}^{n}\left( \frac{\partial}{\partial z^{j}} -\sqrt{-1}h_{j}\frac{\partial}{\partial x}\right)\otimes dz^{j} -  \sum_{j=1}^{n}\left( \frac{\partial}{\partial \bar{z}^{j}} +\sqrt{-1}h_{\bar{j}}\frac{\partial}{\partial x}\right)\otimes d\bar{z}^{j}\right\} $
\item[\textbullet] $g= \eta \otimes \eta + 2\sum_{j,l=1}^{n}h_{j\bar{l}}dz^{j}d\bar{z}^{l}$
\item[\textbullet] $D\otimes \mathbb{C}$ is spanned by
\begin{equation*}\label{preferred basis}
X_{j}:= \pl{z^{j}}-\sqrt{-1}h_{j}\pl{x} \quad X_{\bar{j}}:= \pl{\bar{z}^{j}}+\sqrt{-1}h_{\bar{j}}\pl{x}
\end{equation*}
\end{enumerate}
where $h: U\rightarrow \mathbb{R}$ is a local, basic function (ie.  $\frac{\partial}{\partial x} h=0$), and we have used $h_{j} = \frac{\partial }{\partial z^{j}}h$, and $h_{j\bar{l}} =  \frac{\partial^{2} }{\partial z^{j}\partial \bar{z}^{l}}h$.  See, for example, \cite{GKN}.  We can additionally assume that in these coordinates $h_{j}(p)$=0.  
\begin{versionb}
We define the transverse Christoffel symbols by
\begin{equation*}
\nabla^{T}_{X_{A}}X_{B} = \tilde{\Gamma}_{AB}^{C}X_{C}
\end{equation*}
\end{versionb}
An important observation is that the transverse Christoffel symbols with mixed barred and unbarred indices are identically zero.  Moreover, the pure barred and unbarred Christoffel symbols are given by the familiar formula from K\"ahler geometry
\begin{equation}\label{local christ symbols}
\tilde{\Gamma}_{ij}^{k} = (g^{T})^{k\bar{p}}\partial_{i}g^{T}_{\bar{p} j}.
\end{equation}
\begin{versionb}
Using these equations and the familiar computations from K\"ahler geometry we see that the transverse Riemann tensor satsifies
\begin{equation*}
Rm^{T}_{kl} := -[\nabla^{T}_{X_{k}}, \nabla^{T}_{X_{l}}] = 0 \quad \text{ and }\quad Rm_{\bar{k}\bar{l}}^{T}:=-[\nabla^{T}_{X_{\bar{k}}}, \nabla^{T}_{X_{\bar{l}}}]=0
\end{equation*}
\end{versionb}
Throughout this paper we will refer to these coordinates as \emph{preferred local coordinates}.  

\begin{definition}
A $p$-form $\alpha$ on $(S,\xi, \eta, \Phi, g)$ is called basic if $\iota_{\xi}\alpha=0$, and $\mathcal{L}_{\xi}\alpha =0$.  In particular, a function $f$ is said to be basic if $\mathcal{L}_{\xi}f =0$.  The ring of smooth basic functions will be denoted by $C^{\infty}_{B}$; it is clearly a sub-ring of $C^{\infty}$.
\end{definition}
\begin{versionb}
Let $\Lambda_{B}^{p}$ be the sheaf of basic $p$-forms, and $\Omega^{p}_{B} = \Gamma(S, \Lambda_{B}^{p})$ the global sections.  It is clear that the de Rham differential $d$ preserves basics forms, and hence restricts to a well defined operator $d_{B}:\Lambda^{p}_{B}\rightarrow \Lambda^{p+1}_{B}$.  We get a complex
\begin{equation*}
0\rightarrow C^{\infty}_{B}(S)\rightarrow \Omega^{1}_{B}\xrightarrow{d_{B}}\cdots \xrightarrow{d_{B}} \Omega^{2n}_{B} \xrightarrow{d_{B}} 0
\end{equation*}
whose cohomology groups, denoted by $H^{p}_{B}(S)$, are the basic de Rham cohomology groups.  
\end{versionb}
The transverse complex structure $\Phi$ allows us to decompose $\Lambda^{r}_{B}\otimes \mathbb{C} = \mathop{\oplus}_{p+q=r} \Lambda^{p,q}_{B}$.  We can then decompose $d_{B} = \partial_{B} + \overline{\partial}_{B}$, where $\partial_{B}: \Lambda^{p,q}_{B}\rightarrow \Lambda^{p+1,q}_{B}$ and $\overline{\partial}_{B}:\Lambda^{p,q}_{B}\rightarrow \Lambda^{p,q+1}_{B}$.  Any form $\alpha \in \Lambda^{p,q}$ which satisfies $\iota_{\xi}\alpha =0$ can naturally be regarded as an element of $\Lambda^{p,q}Q^{*}$.  The extra condition that $\mathcal{L}_{\xi} \alpha=0$ ensures that $d\alpha \in \Lambda^{p+q+1}Q^{*}$.  In particular, we can regard $\dbar_{B}$ on basic forms as a map $\dbar_{B}:\Lambda^{p,q}Q^{*} \rightarrow \Lambda^{p,q+1}Q^{*}$, and similarly for $\dop{}$.  The transverse metric induces $L^{2}$ inner products on each of these spaces and so, as in the K\"ahler case, we may define the $\dop{B}$ and $\dbar_{B}$ Laplacians on the bundles $\Lambda^{p,q}_{B}$.  For example, the $\dbar_{B}$ Laplacian is given by $ \square_{B} = (g^{T})^{k\bar{j}}\nabla^{T}_{k}\nabla^{T}_{\bar{j}}$.  We now say a brief word about volume forms and integration by parts on Sasaki manifolds.  As noted before, the form $(\frac{1}{2}d\eta)^{n} \wedge \eta$ is a non-vanishing $(2n+1)$-form, and hence defines a volume form on $S$.  One can check that it agrees with the standard Riemannian volume form by looking in preferred local coordinates.  Furthermore, in preferred local coordinates we have
\begin{equation*}
 d\mu := (\frac{1}{2}d\eta)^{n}\wedge \eta  = (\sqrt{-1})^{n}n!\det(g^{T}_{\bar{k}j})dz^{1}\wedge d\bar{z}^{1}\wedge \dots \wedge dz^{n}\wedge d\bar{z}^{n} \wedge dx.
\end{equation*}
Combining this with the local formula~(\ref{local christ symbols}) for the Christoffel symbols of $Q$, we note that the standard integration by parts formulae from K\"ahler geometry hold for the bundle $Q$, its duals, tensor powers and wedge products.  For example, we have;

\begin{prop}
Let $\phi \in \Lambda^{p,q-1}_{B}$, $\psi \in \Lambda^{p,q}_{B}$ be two basic forms, the we have
\begin{equation*}
\int_{S}(g^{T})^{i\bar{j}}\nabla^{T}_{\bar{j}}\phi_{\bar{B}A} \overline{\psi_{\bar{i}\bar{D}C}}(g^{T})^{\bar{C}A}(g^{T})^{\bar{D}B} d\mu = -\int_{S}\phi_{\bar{B}A} \overline{(g^{T})^{j\bar{i}}\nabla^{T}_{j}\psi_{\bar{i}\bar{D}C}}(g^{T})^{\bar{C}A}(g^{T})^{\bar{D}B} d\mu 
\end{equation*}
\end{prop}

The proof of this proposition follows from the standard computation in the K\"ahler setting, and so we omit the details.  The interested reader can find the computation in \cite{WHe}.
\begin{versionb}
Consider the form $\rho^{T}= Ric^{T}(\Phi \cdot, \cdot)$, which is called the \emph{transverse Ricci form}.  In analogy with the K\"ahler case we have
\begin{equation*}
\rho^{T} = -\partial_{B} \overline{\partial}_{B} \log \det (g^{T}),
\end{equation*}
and hence $\rho^{T}$ defines a basic cohomology class, $[\rho^{T}]_{B}$, which is called the \emph{basic first Chern class}, and is independent of the transverse metric.  
\end{versionb}
A transverse metric $g^{T}$ is called a \emph{transverse Einstein metric} if it satisfies $Ric^{T}=\kappa g^{T}$ for some constant $\kappa$.  In particular, we observe that it is necessary that the basic first Chern class be signed.  The following proposition describes a well-known obstruction to the existence of a transverse Einstein metric.  See, for example, \cite{Futaki}.
\begin{prop}
The basic first Chern class is represented by $\kappa d\eta$ for some constant $\kappa$ if and only if $c_{1}(D)=0$, where $D = \ker \eta$ is the contact subbundle.
\end{prop}

If $g^{T}$ is a transverse Einstein metric with Einstein constant $\kappa$, then one can check that the Sasaki metric $g$ is $\eta$-\emph{Einstein}, satisfying $$Ric_{g} = (\kappa-2) g + (2n+2-\kappa)\eta \otimes \eta.$$ The metric $g$ is \emph{Sasaki-Einstein} if $Ric_{g} = \lambda g$.  In particular, a Sasaki manifold admits a Sasaki-Einstein metric if and only if it admits a transverse Einstein metric with Einstein constant equal 2n+2.  In fact, we shall see that if $\kappa >-2$, then any transverse Einstein metric can be deformed to a Sasaki-Einstein metric via the so called $\mathcal{D}$-homothetic transformations. 

\section{Perturbations of Sasaki Structures and the Sasaki-Ricci Flow}
A Sasaki manifold has a large number of defining structures.  It is natural to ask what happens when one perturbs these structures, individually or together.  In this section we describe two types of deformations of Sasaki structures, the second of which motivates the Sasaki-Ricci flow. We begin by describing the $\mathcal{D}$-homothetic deformations introduced by Tanno \cite{Tanno}.  For $a>0$, the rescaling
\begin{equation*}
g'=ag+(a^{2}-a)\eta\otimes\eta, \text{   } \eta'=a\eta, \text{   }, \xi' = a^{-1}\xi, \text{   } \Phi'=\Phi
\end{equation*}
gives a Sasaki structure $(\xi',\eta', \Phi', g')$ with the same holomorphic structure on the cone, but with radial variable $r'=r^{a}$.  One can check that if $g^{T}$ is a transverse Einstein metric with Einstein constant $c>0$, then the  $\mathcal{D}$-homothetic deformation with $a=c/2n$ gives a Sasaki-Einstein metric.  Throughout this paper, we shall assume that $c_{1}^{B}(S)>0$, and  $c_{1}(D)=0$.  By making a $\mathcal{D}$-homothetic deformation we may always assume that $(2n+2)[\frac{1}{2} d\eta]_{B} = c_{1}^{B}(S)$.  A second class of transformations is described in the following proposition, which is proved in \cite{Futaki, Sparks}.

\begin{prop}\label{perturb prop}
Fix a Sasaki manifold $\mathcal{S} = (S,g,\eta, \xi, \Phi)$.  Then any other Sasaki structure on S with the same Reeb vector field $\xi$, the same complex structure on the cone $C(S) = \mathbb{R} \times S$, and the same transversely complex structure on the Reeb foliation is related to the original structure via the deformed contact form $\eta'=\eta + 2d_{B}^{c}\phi$, where $\phi$ is a smooth basic function that is sufficiently small.
\end{prop}

Deformations of this type shall be referred to as deformations of type II, following \cite{BoyGalMat}.  This motivates the definition of the space of Sasaki potentials;

\begin{definition}
Fix a Sasaki structure $(\xi, \eta, \Phi)$.  We define the space of Sasaki potentials to be
\begin{equation*}
\mathcal{H}_{\eta} = \{\phi \in C^{\infty}_{B} : \tilde{\eta} = \eta +2d_{B}^{c}\phi \text{ is a contact 1-form}\}.
\end{equation*}
If $\eta$ and $\eta'$ are related by some $\phi \in \mathcal{H}_{\eta}$, then we will say that $\eta'$ is in the K\"ahler class of $\eta$.
\end{definition}
Notice that under type II deformations the transverse metric is deformed by $\tilde{g}^{T} = g^{T} + d_{B}d_{B}^{c}\phi$.  In \cite{SmoWaZa} the flow
\begin{equation}\label{SRF}
\frac{\partial g^T}{\partial t} = -Ric_{g(t)}^{T} + \kappa g^{T}(t),
\end{equation}
was studied.  We will refer to this flow as the \emph{normalized} Sasaki Ricci flow, as the volume is fixed under the flow.  When the initial metric satisfies $\kappa g^{T} \in c_{1}^{B}(S)$, it was shown that this flow can be reduced to a transversely parabolic complex Monge-Amp\`ere equation on the Sasaki potential, given by
\begin{equation}\label{PCMA}
\pl{t}\phi = \log \det(g_{\bar{k}l}^{T}+\partial_{l}\partial_{\bar{k}}\phi)- \log \det(g_{\bar{k}l}^{T})+\kappa \phi -F,
\end{equation}
where the function $F$ is defined by $Ric^{T} = \kappa d\eta_{o} + d_{B}d_{B}^{c}F$, according the transverse $\dop{}\dbar{}$-lemma of \cite{ElKac}. It was proved in \cite{SmoWaZa} that~(\ref{SRF}) is well-posed. By Proposition~\ref{perturb prop}, the solution to equation~(\ref{PCMA}) defines a one parameter family of Sasaki structures with the same transverse complex structure, the same Reeb field and the same complex structure on the cone.  Moreover, it was shown that the solution $\phi$ exists for all time, remains basic, and converges exponentially if $c^{1}_{B} \leq 0$. 
The transverse $\dop{}\dbar$-lemma of \cite{ElKac} implies there is a basic function $u:S\rightarrow \mathbb{R}$ so that
\begin{equation*}
\dop{j}\dbop{k}\dot{\phi} = \dot{g}^{T}_{\bar{k}j} = -R^{T}_{\bar{k}j} + g^{T}_{\bar{k}j} = \dop{j}\dbop{k}u,
\end{equation*}
and so $\phi$ evolves by $\dot{\phi}(t) = u(t) + c(t)$.  We can use the function $c(t)$ to adjust the initial value $\phi(0)$.  As in the K\"ahler case, the specific choice for the initial value of $\phi$ plays an important role in translating estimates for the transverse K\"ahler potential to estimates for the transverse metric $g^{T}(t)$, \cite{PSS}.  Below, we make the case that a similar choice for the initial value of $\phi$ as in \cite{PSS} is preferred in the Sasaki-Ricci flow.  We first compute that
\begin{equation*}
\dot{u} = \square_{B}u+u +a(t)
\end{equation*}
for some basic function $a(t)$ depending only on $t$, which we fix by $\int e^{-u}d\mu = Vol(S)$.  It will simplify matters to define a probability measure on $S$ by $d\rho :=Vol(S)^{-1} e^{-u} d\mu$.  One can easily compute that
\begin{equation}\label{a def}
a(t) = \int_{S}u d\rho = \frac{1}{Vol(S)}\int_{S} ue^{-u}d\mu.
\end{equation}

In \cite{TristC} (see also \cite{WHe} for another approach) transverse $\mathcal{W}$ and $\mu$ functionals were introduced, and shown to be monotone along the flow, thereby opening Perelman's methods to the Sasaki setting.  The author applied these functionals to obtain a non-collapsing theorem for the Sasaki-Ricci flow, and to extend Perelman's uniform estimates for the K\"ahler-Ricci flow to the Sasaki setting.  The precise results are;

\begin{prop}[\cite{TristC} Proposition 7.1]\label{non-collapse 1}
Let $g^{T}(s)$ be a solution of the normalized Sasaki-Ricci flow, and let $\rho>0$.  There exists a constant $c>0$, depending only on $g(0)$, and $\rho$ such that for every $p \in S$, and $t\geq 0$
\begin{equation*}
\int_{\{y : d^{T}(p,y)<r\}}d\mu > cr^{2n}
\end{equation*}
\end{prop}

\begin{Theorem}[\cite{TristC} Theorem 1.3]\label{Pman thm}
Let $g(t)$ be a solution of the normalized Sasaki-Ricci flow on a compact Sasaki manifold $(S,\xi)$ of real dimension $2n+1$, and transverse complex dimension n, with $c^{1}_{B}(S)>0$.  Let $u \in C^{\infty}_{B}(S)$ be the transverse Ricci potential.  Then there exists a uniform constant $C$, depending only on the initial metric $g(0)$ so that
\begin{equation}
|R^{T}(g(t))| + |u|_{C^{1}} + diam^{T}(S,g(t)) < C
\end{equation}
where $diam^{T}(S,g(t)) = \sup_{x,y\in S} d^{T}(x,y)$.
\end{Theorem} 

It was pointed out in \cite{WHe} that the transverse diameter bound in Theorem~\ref{Pman thm} implies a bound for the diameter.  We include the argument for completeness.

\begin{Lemma}\label{diameter bound lemma}
There exists a constant $C>0$, such that $diam(S,g(T))<C$ for all $t$ along the Sasaki-Ricci flow.
\end{Lemma}
\begin{proof}
For every $t\geq 0$, we have $g(t)(\xi,\xi) =1$.  In particular, if $p \in S$ has $orb_{\xi}p$ closed, then the length of the curve defined by $orb_{\xi}p$ is independent of $t$.  Thus, the result is obvious in the regular and quasi-regular cases.  When the Sasaki structure is irregular, the results of Rukimbira \cite{Ruk, Ruk1} imply that there exists at least $n+1$ closed orbits of the Reeb field.  Let $p\in S$ be a point with closed orbit, and let $orb_{\xi}p$ have length $A$.  We have
\begin{equation*} 
d(x,y) \leq d^{T}(x,p) +A + d^{T}(y,p) \leq 2diam^{T}(S,g(t)) + A.
\end{equation*} 
which completes the proof.
\end{proof}

We now discuss the choice for the initial value of the transverse K\"ahler potential.  Suppose that a given flow $\phi$ satisfies $\phi(0)=c_{0}$, then one can easily check that $\tilde{\phi} := \phi + (\tilde{c}_{0} -c_{0})e^{\kappa t}$ satisfies the same flow with initial condition $\tilde{c_{0}}$.  This underlines the importance of choosing the initial value properly; any two solutions with different initial value differ by terms diverging exponentially in time.  We introduce the quantity
\begin{equation}\label{initial condition}
c_{0} = \int_{0}^{\infty} e^{-t}\|\nabla \dot{\phi} \|_{L^{2}}^{2} dt + \frac{1}{Vol(S)}\int_{S}u(0) d\mu_{0}.
\end{equation}
One easily checks that this does not depend on $\phi(0)$.  The first indication that this is the correct choice for $\phi(0)$ is that the bound in Theorem~\ref{Pman thm} implies
\begin{equation}\label{uniform bound for phi}
\sup_{t \geq 0} \|\dot{\phi}\|_{C^{0}} \leq C.
\end{equation}
To see this, observe that $\dop{B}\dbar_{B} (u-\dot{\phi}) =0$, and so by the uniform bound for $u$ it suffices to  bound, $\alpha(t) := Vol(S)^{-1} \int_{S} \dot{\phi} d\mu_{t}$.  This is easily done by computing the evolution of $\alpha$; see the computations in \cite{PSS}.  The upshot of this is contained in the following;

\begin{prop}\label{uniform Yau}
Let $(S, \xi, \eta_{0}, \Phi, g_{0})$ be a compact Sasaki manifold with $\kappa [\frac{1}{2}d\eta_{0}]_{B} =c_{1}^{B}(S)$ for any constant $\kappa$.  Consider the Sasaki-Ricci flow defined by~(\ref{PCMA}).  The we have the a priori estimates
\begin{equation*}
\sup_{t\geq 0} \|\phi\|_{C^{0}} \leq A_{0} <\infty \iff \sup_{t\geq 0} \| \phi \|_{C^{k}} \leq A_{k} <\infty \quad \forall k \in \mathbb{N}
\end{equation*}
\end{prop}

The estimates in Proposition~\ref{uniform Yau} are the transverse parabolic version of Yau's famous estimates \cite{Yau2}, or equivalently the parabolic version of El Kacimi-Alaoui's generalization of Yau's estimates \cite{ElKac}.  For the K\"ahler-Ricci flow, these estimates are well known, and can be found, for example, in \cite{Cao, PSS}.  For the Sasaki-Ricci flow, the estimates can be found in \cite{SmoWaZa}.  In light of the uniform bound for $\dot{\phi}$, it is straight forward to check that Proposition~\ref{uniform Yau} holds.

\section{Bernstein-Bando-Hamilton-Shi Estimates for the Transverse Curvature}

In this section we prove Theorem~\ref{BBS Thm}, which follows essentially from the arguments in the Riemannian case, and the curvature identities~(\ref{curvature relation}), (\ref{curvature relation 2}), and (\ref{curvature relation 3}).  As the techniques used to prove this theorem in the Riemannian case are purely local, one hopes that this result will carry over to the Sasaki case.  However, we must proceed with some care.  For example, the standard commutation formulae for the covariant derivative and the Laplacian do not apply to `tensors' on the bundle $Q$.  In fact, the commutation relation which does hold involves not only the transverse Riemann tensor, but also the full Riemann curvature; here  the curvature identities~(\ref{curvature relation}), (\ref{curvature relation 2}), and (\ref{curvature relation 3}) are crucial.  In order to avoid being swamped by indices we introduce the following notation; if $A$ and $B$ are two sections of $TS^{*\otimes p} \otimes Q^{*\otimes q}$,  we denote by $A*B$ any quantity obtained from $A\otimes B$ by summation over paired indices, contraction with $g, g^{-1}, g^{T}$, or $(g^{T})^{-1}$, and multiplication by constants depending only on $n, p$ and $q$.

\begin{Lemma}\label{commutator lemma}
Suppose $A \in TS^{*\otimes p} \otimes Q^{*\otimes q}$. Then there is a constant $C$, depending only on $p,q$, and $n$, such that the following commutation relation holds;
\begin{equation*}
[\nabla, \Delta]A = Rm^{T}*\nabla A + A*\nabla Rm^{T} +  C\nabla A.
\end{equation*}
\end{Lemma}
\begin{proof}
The argument is elementary, and so we only provide a sketch.  The key point is that commuting covariant derivatives yields an expression involving both the full Riemann tensor, and the transverse Riemann tensor.  In particular,
\begin{equation*}
[\nabla, \Delta]A = Rm^{T} * \nabla A + Rm*\nabla A + A * \nabla Rm + A*\nabla Rm^{T}.
\end{equation*}
We now use the curvature relations~(\ref{curvature relation}), (\ref{curvature relation 2}), and (\ref{curvature relation 3}) to replace all the terms involving the full curvature tensor with terms involving only the metric and $Rm^{T}$.  Collecting terms, and observing that any constants which appear depend only on $n$, $p$ and $q$, the lemma is proved.
\end{proof}

\begin{proof}[Proof of Theorem~\ref{BBS Thm}]
We begin by computing
\begin{versionb}
\begin{equation*}
\RmT {j}{i}{l}{k} = -[ \nabla^{T}_{\bar{j}}, \nabla^{T}_{i}]^{l}{}_{k} = -\dbop{j}\left((g^{T})^{l\bar{p}}\dop{i}g^{T}_{\bar{p}k}\right)
\end{equation*}
We suppress the symbol $T$ in what follows, and compute 
\begin{equation*}
\begin{aligned}
\pl{t} \RmT {j}{i}{l}{k} &= -\dbop{j}\left(g^{l\bar{s}}R_{\bar{s}m}g^{m\bar{p}}\dop{i}g_{\bar{p}k} - g^{l\bar{p}}\dop{i}\left(R_{\bar{p}k}\right)\right)\\
&= \dbop{j}\left(g^{l\bar{s}} \left(\dop{i}R_{\bar{s}k} - \tilde{\Gamma}_{ik}^{m}R_{\bar{s}m}\right)\right)\\
&= \nabla_{\bar{j}}\nabla_{i}R^{l}{}_{k}
\end{aligned}
\end{equation*}
We now compute
\begin{equation*}
\begin{aligned}
\Delta \Rm{\bar{j}}{i}{l}{k} &= g^{p\bar{q}}\nabla_{p}\nabla_{\bar{q}}\Rm{\bar{j}}{i}{l}{k} \\
&= g^{p\bar{q}}\nabla_{p}\nabla_{\bar{j}}\Rm{\bar{q}}{i}{l}{k}\\
&=g^{p\bar{q}}\nabla_{\bar{j}}\nabla_{p}\Rm{\bar{q}}{i}{l}{k} + Rm*Rm\\
&=\nabla_{\bar{j}}\nabla_{i}R^{l}{}_{k} + Rm*Rm,
\end{aligned}
\end{equation*}
where in the second and the fourth lines we have used the K\"ahler property for the transverse metric to interchange indices.  We thus obtain
\end{versionb}
\begin{equation}\label{RmT evolution}
\pl{t}Rm^{T} -\Delta Rm^{T} = Rm^{T}*Rm^{T}.
\end{equation}
We can now compute the evolution equation for $|\nabla Rm^{T}|^{2}$.  Before proceeding, we point out that the quantity  $|\nabla Rm^{T}|^{2}$ involves both $g^{T}$ and $g$, as the covariant derivative $\nabla$ takes arguments in $TS$.  However, by looking in preferred coordinates we clearly have $
\nabla_{\pl{x}}Rm^{T} = 0,$ and so we see that the norm $|\nabla Rm^{T}|^{2}$ agrees with the norm when we replace $g$ by $g^{T}$, regarded as a bilinear form on $TS$.  With this in mind, we obtain
\begin{equation}\label{grad RmT evolution}
\pl{t}|\nabla Rm^{T}|^{2}= \Delta|\nabla Rm^{T}|^{2} -2 |\nabla^{2} Rm^{T}|^{2} + Rm^{T}*(\nabla Rm^{T})^{*2}+ (\nabla Rm^{T})^{*2}.
\begin{versionb}
\begin{aligned}
\pl{t}|\nabla Rm^{T}|^{2} &= 2\left\langle \nabla\left(\pl{t}Rm^{T}\right), \nabla Rm^{t}\right\rangle\\ &+ \nabla Ric^{T}*Rm^{T}*\nabla Rm^{T} + Rc^{T}*\nabla Rm^{T}*\nabla Rm^{T}\\
&=\Delta|\nabla Rm^{T}|^{2} -2 |\nabla^{2} Rm^{T}|^{2} + Rm^{T}*(\nabla Rm^{T})^{*2}+ (\nabla Rm^{T})^{*2}.
\end{aligned}
\end{versionb}
\end{equation}
The last line follows by Lemma~\ref{commutator lemma}.  We now consider the quantity $F = t|\nabla Rm^{T}|^{2} + \beta |Rm^{T}|^{2}$.  Using equations~(\ref{RmT evolution}) and~(\ref{grad RmT evolution}), we compute that
\begin{equation*}
\begin{aligned}
\pl{t}F &\leq |\nabla Rm^{T}|^{2} + t\pl{t}|\nabla Rm^{T}|^{2} + \beta (Rm^{T})^{*3} + \beta \Delta |Rm^{T}|-2\beta|\nabla Rm^{T}|^{2}\\
&\leq \Delta F + \left(1+c_{1}t|Rm^{T}| +c_{2}t-2\beta\right)|\nabla Rm^{T}|^{2} +c_{3}|Rm^{T}|^{3}.
\end{aligned}
\end{equation*}
By assumption, $|Rm^{T}| \leq K$ if $t \in[0,\alpha/K]$.  Set $2\beta = 1+ c_{1}\alpha + c_{2}\frac{\alpha}{K}$, then
\begin{equation*}
\pl{t}F -\Delta F \leq c_{3}K^{3}.
\end{equation*}
Applying the maximum principle yields $F(x,t) \leq \beta K^{2} + c_{3}\beta K^{3}t$.  Using the definition of $\beta$ we obtain that there is a constant $C_{4}$ depending only on $n$, and  $\max\{\alpha, 1\}$, such that for every $t \in (0,\alpha/K]$ we have
\begin{equation*}
|\nabla Rm^{T}| \leq \sqrt{\frac{F}{t}} \leq C_{4}t^{-1/2}\max\{K^{1/2}, K\}.
\end{equation*}
This proves the theorem in the case $m=1$.  The general case of $m>1$ follows by making similar adaptations to the K\"ahler, or Riemannian case.  See \cite{BenChow} for details in these cases.  The curvature equations~(\ref{curvature relation}), (\ref{curvature relation 2}), and (\ref{curvature relation 3}) show that our bounds for $Rm^{T}$ extend to bounds for the full Riemann tensor.
\begin{versionb}
 We now show that our bounds for $Rm^{T}$ extend to bounds for the full Riemann tensor.  Fix a point $p\in S$, and preferred local coordinates in an open neighbourhood of $U \ni p$, centered at $p$.  Using the curvature equations~(\ref{curvature relation}), (\ref{curvature relation 2}), and (\ref{curvature relation 3}), one can easily check that for the basis of the tangent space in~(\ref{preferred basis}) we have
\begin{equation}\label{local curv relations}
\begin{aligned}
&Rm_{\bar{k}j\bar{l}p} = Rm^{T}_{\bar{k}j\bar{l}p} - g^{T}_{\bar{k}p}g^{T}_{\bar{l}j}\\
& Rm = 0 \text{ for any other set of indices }
\end{aligned}
\end{equation}
It is immediate that the bound for $Rm^{T}$ at $p$ implies a bound for $Rm$ at $p$.  Moreover, since equation~(\ref{local curv relations}) holds in a neighbourhood of $p$, and the transverse metric is covariantly constant, we see that the $C^{k}$ bounds for $Rm^{T}$ imply $C^{k}$ bounds for the full Riemann tensor.  The proof is complete.
\end{versionb}
\end{proof}

\begin{corollary}\label{uniform curvature bound}
If $Rm^{T}$ is uniformly bounded in $C^{0}$ along the normalized Sasaki-Ricci flow, then for any time $A>0$, there is a constant $C_{A,k}$ depending only on $|Rm^{T}|_{C^{0}}$, $k$ and $A$, so that $|Rm^{T}|_{C^{k}} \leq C_{A,k}$ for all $t \geq A$.
\end{corollary}
\begin{versionb}
\begin{corollary}\label{flow compactness}
If $(S,\xi,\eta, \Phi, g_{0})$ is a Sasaki manifold, and $g(t)$ is a solution of the normalized Sasaki-Ricci flow with $|Rm^{T}|<C$ uniformly along the flow, then there exists a subsequence $t_{k} \rightarrow \infty$, and diffeomorphisms $F_{t_{k}}$ such that $(S, F_{t_{k}}^{*}g(t_{k}))$ converge in $C^{\infty}$ to a Sasaki manifold $(S, g_{\infty})$.
\end{corollary}
\begin{proof}
Corollary~\ref{uniform curvature bound} and the diameter bound in Lemma~\ref{diameter bound lemma}, combined with the volume preservation along the flow imply by \cite{Cheeger} that the injectivity radius is uniformly bounded below.  Hence, $S$ has uniformly bounded geometry in the $C^{\infty}$ sense.  By the compactness result of Hamilton \cite{RHam}, there is a subsequence $\{t_{k}\}$ and diffeomorphisms $F_{t_{k}}$ such that $(S,F^{*}_{t_{k}}g(t_{k}))$ converges in $C^{\infty}$.  It is easy to check that this limit is Sasaki.
 \end{proof}
\begin{remark}
For the dubious reader, we have included the proof of the Sasaki version of Hamilton's compactness theorem in \S8.
\end{remark}
\end{versionb}
\section{The $\dbar$-operator on Foliate Vector Fields and Stability on Sasaki Manifolds}

For the remainder of the paper we will be concerned with employing various types of stability to prove the convergence of the Sasaki-Ricci flow.  We begin by presenting two notions of stability on Sasaki manifolds which generalize notions on K\"ahler manifolds.  Later, we shall introduce the sheaf of transverse foliate vector fields, which seems central to the problem of stability on Sasaki manifolds.  We first define the Futaki invariant of a Sasaki manifold.  In the Sasaki case, the Lie algebra on which the Futaki invariant acts is the space of Hamiltonian, holomorphic vector fields.

\begin{definition}[\cite{Futaki} Definition 4.5]\label{ham hol def}
Let $U_{\alpha} = I \times V_{\alpha}$ be a foliated coordinate patch, with $I \subset \mathbb{R}$ an open interval, and $V_{\alpha} \subset \mathbb{C}^{n}$.  Let $\pi_{\alpha}:U_{\alpha}\rightarrow V_{\alpha}$ be the projection.  A complex vector field $X$ on a Sasaki manifold is called a  Hamiltonian holomorphic vector field if 
\begin{enumerate}
\item[{\it(i)}] $d\pi_{\alpha}$ is a holomorphic vector field on $V_{\alpha}$
\item[{\it (ii)}] the complex valued function $u_{X} := \sqrt{-1}\eta(X)$ satisfies
\begin{equation*}
\dbop{B}u_{X} = -\frac{\sqrt{-1}}{2}\iota_{X}d\eta.
\end{equation*}
\end{enumerate}
Such a function $u_{X}$ is called a Hamiltonian function.
\end{definition}
\begin{remark}
If $X$ is a Hamiltonian holomorphic vector field, then in preferred local coordinates
\begin{equation}\label{local ham hol}
X = \eta(X)\pl{x} + \sum_{i=1}^{n} X^{i}\pl{z^{i}} - \eta\left( \sum_{i=1}^{n}  X^{i}\pl{z^{i}}\right) \pl{x},
\end{equation}
where the $X^{i}$ are local, holomorphic basic functions.
\end{remark}

The Futaki invariant was originally defined for Sasaki manifolds by Boyer, Galicki and Simanca in \cite{BoyGalSim}, where it was considered to be a character on a quotient of the Lie algebra of ``transversally holomorphic" vector fields (see \cite{BoyGalSim} Definition 4.5).  In \cite{Futaki} Futaki, Ono and Wang recast the Futaki invariant as a character of the Lie algebra on Hamiltonian holomorphic vector fields as defined above.  For our purposes, we are only concerned with the case when the distribution $D$ has $c_{1}(D) =0$.

\begin{Theorem}[\cite{BoyGalSim} Proposition 5.1, \cite{Futaki} Theorem 4.9]
Let $(S,g)$ be a Sasaki manifold with $c_{1}^{B}(S)>0$, and $c_{1}(D) =0$.  Assume $ (2n+2)g^{T} \in c_{1}^{B}(S)$ and let $u$ be the transverse Ricci potential for the metric $g$, and let $X$ be a holomorphic, Hamiltonian vector field.  We define the Futaki invariant $fut_{S}(X)$ by the equation
\begin{equation*}
fut_{S}(X) := \int_{S} Xu d\mu.
\end{equation*}
Then $fut_{S}(X)$ is independent of the choice of Sasaki metric in $(2n+2)^{-1}c_{1}^{B}(S)$.
\end{Theorem}
\begin{remark}
It is clear that the vanishing of the Futaki invariant is necessary for the existence of a Sasaki-Einstein metric.
\end{remark}

Later in this section we will provide an alternative characterization of the Futaki invariant as a character on a certain subspace of the global sections of the soon-to-be-defined sheaf $\E$, and show that this characterization is equivalent to the above.  We now introduce the Mabuchi energy on a Sasaki manifold in the special case that $c_{1}(D)=0$;

\begin{Theorem}[\cite{Futaki} Theorem 4.12]
Let $(S,g)$ be a Sasaki manifold, with $c_{1}^{B}(S) = \kappa[d\eta_{0}]_{B}$.  Let $\eta'$ be in the K\"ahler class of $\eta_{0}$.  Let $\phi_{t}$, $t\in[a,b]$ be a path in $\mathcal{H}_{\eta_{0}}$  connecting $\eta, \eta'$.  Then the Mabuchi K-energy,
\begin{equation*}
\nu_{\eta_{0}}(\eta') =- \int_{a}^{b}\int_{S}\dot{\phi}_{t}(R^{T}-\bar{R}^{T}) d\mu_{t}dt,
\end{equation*}
is independent of the path $\phi_{t}$. 
\end{Theorem}

In the K\"ahler theory, the boundedness below of the Mabuchi K-energy is crucial to the existence of canonical metrics.  It is known that a bound below for the K-energy is not sufficient to guarantee the existence of a K\"ahler-Einstein metric; a counterexample is given by Tian's unstable deformation of the Mukai-Umemura threefold \cite{Chen, Don, Don1, Tian1}.  However, Bando and Mabuchi proved that any manifold that admits a K\"ahler-Einstein metric  $\omega$ necessarily has the K-energy bounded below on the K\"ahler class of $\omega$ \cite{Bando, BM}.  Their result extends to Sasaki manifolds.

\begin{prop}\label{Mab prop}
Suppose $(S,\xi, \eta, \Phi, g)$ is a Sasaki manifold with  $c_{1}(D)=0$, and $c_{1}^{B}(S) = (2n+2)[\frac{1}{2}d\eta]_{B}$.  Assume that $S$ admits a Sasaki-Einstein metric in the K\"ahler class of $g$.  Then the Mabuchi K-energy is bounded below on $[\frac{1}{2}d\eta]_{B}$.
\end{prop}
\begin{proof}
The proof is a consequence of the work of Nitta and Sekiya \cite{NS}.   The estimates in \cite{NS} imply that the result of Bando and Mabuchi in \cite{BM} holds on Sasaki manifolds.  That is, the Mabuchi K-energy is bounded below on the set of Sasaki metrics with positive transverse Ricci curvature.  It is straight forward to check that the extension of the results of \cite{BM}, due to Bando in \cite{Bando} carries over verbatim to the Sasaki setting.
\end{proof}

\begin{versiona}
The proof of Proposition~\ref{Mab prop} follows essentially from the work of Nitta and Sekiya \cite{NS}, and the papers of Bando, Mabuchi \cite{BM}, and Bando \cite{Bando}, and will be given in the appendix.
\end{versiona}

If $g^{T}(t)$ is evolving by the Sasaki-Ricci flow~(\ref{SRF}), then 
\begin{equation}\label{flow mab}
\nu_{\eta_{0}}(\eta(a))= \int_{0}^{a}\int_{S}u(R^{T}-\kappa n) d\mu_{t} dt=\int_{0}^{a} \int_{S}u(-\square_{B} u) d\mu_{t} dt=\int_{0}^{a} \int_{S} (g^{T})^{j\bar{k}}\dop{j}u\overline{\dop{k}u}d\mu_{t}dt.
\end{equation}
The last term in the above expression is precisely the integral of the $L^{2}$ norm of $\dop{B}u$ regarded as a section of $\Lambda^{1,0}Q^{*}$.  

A significant difficulty in extending the results of \cite{PS} is identifying the appropriate operator to study and  determining the domain on which to study it.  As noted in the introduction, we are guided primarily by the model case of a regular Sasaki manifold.

\begin{definition}
On an open subset $U\subset S$, let $\Xi(U)$ be the Lie algebra of smooth vector fields on $U$ and let $\mathcal{N}_{\xi}(U)$ be the normalizer of the Reeb field in $\Xi (U)$,
\begin{equation*}
\mathcal{N}_{\xi}(U) = \{ X \in \Xi(U) : [X,\xi] \in L_{\xi}\}.
\end{equation*}
We define a sheaf $\mathcal{E}$ on $S$ by
\begin{equation*}
\mathcal{E}(U) := N_{\xi}(U)/L_{\xi}.
\end{equation*}
The sheaf $\E$ will be referred to as the sheaf of transverse foliate vector fields.
\end{definition}

When there is some chance of confusion, for $V\in TS$ we denote by $[V]$ the equivalence class of $V$ in $Q$.  Recalling the exact sequence~(\ref{exact seq}), the inclusion $\mathcal{N}_{\xi}\subset TS$ induces an inclusion of sheaves $\mathcal{E} \subset Q$.  The sheaf $\mathcal{E}$ is easily seen to be a locally free sheaf of $C^{\infty}_{B}$-modules.  When the Reeb field is regular or quasi-regular, the sheaf $\mathcal{E}$ descends through the quotient to the sheaf of smooth sections of the tangent bundle of the K\"ahler manifold or orbifold.  
$\mathcal{E}$ inherits a great deal of structure from the vector bundle $Q$; the metric $g^{T}$ restricts to a metric on $\mathcal{E}$, and it is easy to check that the transverse complex structure $\Phi$ on $Q$ restricts to an endomorphism of $\mathcal{E}$, and hence splits $\mathcal{E}$ as $\mathcal{E} = \mathcal{E}^{1,0} \oplus \mathcal{E}^{0,1}$.  The key observation is that $\mathcal{E}$ has a well defined $\bar{\partial}$ operator.  Define a map $\mathfrak{e} : Q \rightarrow Q$ by $\mathfrak{e}(V) := \nabla^{T}_{\xi} V$. By the definition of $\nabla^{T}$, it is clear that $\mathcal{E}$ is precisely the kernel of the map $\mathfrak{e}$.  In particular, on $\mathcal{E}$ the covariant derivative $\nabla^{T}$ descends to a well defined map $d_{\mathcal{E}} : \mathcal{E} \rightarrow \mathcal{E} \otimes Q^{*}$.  This map is defined for $[V] \in \mathcal{E}$ and $[X] \in Q$ by
\begin{equation*}
d_{\mathcal{E}}[V] ([X]) = \nabla^{T}_{X} [V],
\end{equation*}
and this is clearly independent of the representative of the equivalence class $[X]$.  Moreover, the transverse complex structure $\Phi$ yields a splitting $\mathcal{E} \otimes Q^{*} = \E \otimes (Q^{*})^{1,0} \oplus \E \otimes (Q^{*})^{0,1}$.  The map $d_{\E}$ then splits as $\partial_{\E} + \bar{\partial_{\E}}$.  Hence, we have a well defined operator $\bar{\partial}_{\E} : \E \rightarrow \E \otimes (Q^{*})^{(0,1)}$.  We  can extend the operator $\dbar_{\E}$ to a differential operator.  Dualizing the exact sequence~(\ref{exact seq}) we have an operator, also denoted by $\dbar_{\E}$ satisfying
\begin{equation*}
\dbar_{\E} : \E \rightarrow \E \otimes p^{\dagger}(Q^{*}) \hookrightarrow \E\otimes TS^{*}.
\end{equation*}
\begin{versionb}
\begin{equation*}
0\rightarrow Q^{*} \xrightarrow{p^{\dagger}} TS^{*}\rightarrow L_{\xi}^{*}\rightarrow 0
\end{equation*}
We thus have an operator, also denoted by $\dbar_{\E}$
\begin{equation*}
\dbar_{\E} : \E \rightarrow \E \otimes p^{\dagger}(Q^{*}) \hookrightarrow \E\otimes TS^{*}.
\end{equation*}
\end{versionb}
The metric $g^{T}$ induces $L^{2}$ inner products on $\Gamma(X, \E^{1,0})$ and $\Gamma\left(X, \E^{1,0}\otimes p^{\dagger}(Q^{*})\right)$.  We can then define the formal adjoint of the $\dbar_{\E}$-operator, denoted $\dbar_{\E}^{\dagger}$, on smooth sections by the usual formula.
\begin{versionb}
\begin{equation*}
\langle \dbar_{\E}V, W \rangle_{\E^{1,0}\otimes p^{\dagger}(Q^{*})} = \langle V, \dbar_{\E}^{\dagger}W\rangle_{\E^{1,0}}
\end{equation*}
\end{versionb}
One easily computes using integration by parts that $\bar{\partial}_{\E}^{\dagger} [W] = -(g^{T})^{j\bar{k}} \nabla^{T}_{j}W^{p}_{\bar{k}}$, and hence we can define the Laplacian of $\E$ by
\begin{equation*}
\square_{\E} [V] = -(g^{T})^{j\bar{k}} \nabla^{T}_{j} \nabla^{T}_{\bar{k}}V^{p}.
\end{equation*}
We comment that at first glance this operator does not appear to be elliptic.  However, by recalling the definition of the bundle $\E$, we see that
\begin{equation*}
\square_{\E} [V] = -(\xi^{*}\otimes\xi^{*})\nabla^{T}_{\xi}\nabla^{T}_{\xi}V^{p}  -(g^{T})^{j \bar{k}} \nabla^{T}_{j}\nabla^{T}_{\bar{k}} V^{p}, 
\end{equation*}
which is clearly elliptic.  Thus we can apply the usual elliptic theory.  In a similar fashion we may also define the $\dop{\E}$ Laplacian $\overline{\square}_{E}$.  A standard integration by parts computation proves;

\begin{prop}\label{BK formula}
For every $[V] \in \Gamma(S, \E^{1,0})$, we have the Bochner-Kodaira formula for the Laplacians $\square_{E}$ and $\overline{\square}_{E}$
\begin{equation*}
\| \dop{\E}[V]\|^{2}_{L^{2}} =\| \dbar_{\E} [V] \|^{2}_{L^{2}} + \int_{S}R^{T}_{\bar{k}j}V^{j}\overline{V^{k}} d\mu,
\end{equation*}
\end{prop}
\begin{versionb}
\begin{proof}
We integrate by parts twice to compute
\begin{equation*}
\begin{aligned}
\langle \dop{\E}V, \dop{\E} V\rangle& = \int_{S}(g^{T})^{j\bar{k}} \nabla^{T}_{j}V^{l}g^{T}_{\bar{p} l} \overline{\nabla^{T}_{k}V^{p}} d\mu\\ &= - \int_{S} g^{T}_{\bar{p}l} V^{l}\dop{j}\left(\overline{(g^{T})^{k\bar{j}}\nabla^{T}_{k}V^{p}}d\mu \right)\\ &= \int_{S} \nabla^{T}_{\bar{k}}V^{l}g^{j\bar{k}}(g^{T})_{\bar{p} l} \overline{\nabla^{T}_{\bar{j}}V^{p}} d\mu + \int_{S}R^{T}_{\bar{k}j}V^{j}\overline{V^{k}}d\mu.
\end{aligned}
\end{equation*}
\end{proof}

We now identify the kernel of $\square_{\E}$ restricted to global sections of $\E^{1,0}$.
\end{versionb}
\begin{versionb}
\begin{definition}
We define the space $H^{0}(\E^{1,0})$, which we refer to as the space global holomorphic sections of the sheaf of transverse foliate vector fields, by
\begin{equation*}
H^{0}(\E^{1,0}):=Ker \bar{\partial}_{\E} \big|_{\E^{(1,0)}}.
\end{equation*}
\end{definition}
\end{versionb}

\begin{prop}\label{kernel prop}
We define the space $H^{0}(\E^{1,0})$, which we refer to as the space global holomorphic sections of the sheaf of transverse foliate vector fields, by
\begin{equation*}
H^{0}(\E^{1,0}):=Ker \bar{\partial}_{\E} \big|_{\E^{(1,0)}}.
\end{equation*}
\begin{enumerate}
\item[\textbullet] $H^{0}(\E^{1,0})$ has the structure of a finite dimensional Lie algebra over $\mathbb{C}$.
\item[\textbullet] $H^{0}(\E^{1,0})$ is isomorphic as a Lie algebra to the Lie algebra of holomorphic, Hamiltonian vectorfields on $S$.  
\item[\textbullet] The space  $H^{0}(\E^{1,0})$ depends only on the complex structure $J$ on the cone, and the Reeb field $\xi$, and the transverse holomorphic structure.  In particular, $\dim H^{0}(\E^{1,0})$ is invariant along the Sasaki Ricci flow.
\end{enumerate}
\end{prop}
\begin{remark}
The Lie algebra $H^{0}(\E^{1,0})$ has appeared in the literature before, under a number of different guises.  In \cite{NishTond} it was proved that if the transverse scalar curvature of the Sasaki metric $g$ is constant, then $H^{0}(\E^{1,0})$ is reductive.  $H^{0}(\E^{1,0})$ also played an important role in the work of Boyer, Galicki and Simanca on extremal Sasaki metrics \cite{BoyGalSim}. In \cite{NS} it was shown that if $S$ admits a Sasaki-Einstein metric $g_{SE}$, $\mathcal{G}$ is the identity component of the automorphism group of the Sasaki structure, and $\mathcal{O}$ is the orbit of $g_{SE}$ under the action of $\mathcal{G}$, then the tangent space to $\mathcal{O}$ at the point $g_{SE}$ is isomorphic to $H^{0}(\E^{1,0})$.
\end{remark}

We delay the proof of Proposition~\ref{kernel prop} for a moment in order that we may discuss the local structure of $\E$.  We feel this local picture is essential to understanding the structure we are describing abstractly and so we shall be very explicit in this part of the development.  Let $U \subset S$ be an open subset of $S$ on which we have a preferred coordinate system, and suppose that $[V] \in \Gamma(U, \E)$ is a section of $\E$ over $U$.  Over $U$ we can write $V = V^{i}[\dop{z^{i}}] + V^{\bar{i}}[\dop{\bar{z}^{i}}]$.  Let $V$ be any lift of $[V]$ to $\mathcal{N}_{\xi}(U)$.  In preferred local coordinates
\begin{equation*}
V = f\pl{x} +\sum_{i=1}^{n} V^{i}\frac{\partial}{\partial z^{i}} + \sum_{i=1}^{n}  V^{\bar{i}} \frac{\partial}{\partial \bar{z}^{i}}.
\end{equation*}
As $[V] \in \E$, we necessarily have $V^{i}, V^{\bar{i}}$ basic functions.  If we restrict to $[V] \in \E^{1,0}$, then $V^{\bar{i}} =0$.  A basis of $Q^{0,1}$ over $U$ is given by $\{ [\partial_{\bar{z}^{j}} ] \}$ where $1 \leq j \leq n$.    By definition, we compute
\begin{versionb}
\begin{equation*}
\bar{\partial}_{\E}[V] = \nabla^{T}_{\partial_{\bar{z}^{j}}} [V] \otimes  [\partial_{\bar{z}^{j}} ]^{*}=\sum_{i=1}^{n} \partial_{\bar{j}}V^{i} \left[ \frac{\partial}{\partial z^{i}}\right]\otimes [\partial_{\bar{z}^{j}} ]^{*}.
\end{equation*}
Now it is easy to check that $p^{\dagger}([\partial_{\bar{z}^{j}}]^{*}) = d\bar{z}^{j}$, and so
\end{versionb}
\begin{equation}\label{dbar on E}
\dbar_{\E}[V] =\sum_{i=1}^{n} \partial_{\bar{j}}V^{i} \left[\pl{z^{i}}\right] \otimes d\bar{z}^{j}.
\end{equation} 

\begin{proof}[Proof of Proposition~\ref{kernel prop}]
Since $S$ is compact, the kernel of the self-adjoint elliptic operator $\square_{\E}$ is finite dimensional, and hence $H^{0}(\E^{1,0})$ is a finite dimensional vector space over $\mathbb{C}$.  The Jacobi identity shows that the Lie bracket on $\Xi(S)$ descends to a Lie bracket on $\Gamma(U,\E)$.  Moreover, the local formulae show that the bracket preserves the decomposition $\E = \E^{1,0}\oplus \E^{0,1}$.  Computing over an open set shows that the Lie bracket preserves the space $H^{0}(\E^{1,0})$, and establishes the first statement.  We turn our attention to the second and third statements.  Suppose $[V] \in H^{0}(\E^{1,0})$, then equation~(\ref{dbar on E}) shows that condition {\it(i)} of Definition~\ref{ham hol def} is satisfied for any lift $V_{c}:= \sigma(V) +c\xi$ of $[V]$.  It remains to show condition {\it(ii)} holds for some choice of $c$.  Define $\tilde{c} := -\xi^{*}(\sigma([V]))$.
\begin{versionb}
\begin{equation*}
\tilde{c} := \eta \left(\sum_{i=1}^{n} V^{i} \pl{z^{i}}\right) = -\xi^{*}(\sigma([V])).
\end{equation*}
\end{versionb}
Using the formula for $\eta$ in preferred local coordinates yields $
u_{V_{\tilde{c}}} = -\sum_{i=1}^{n} h_{i}V^{i}$.  As $\dbar_{\E}[V] =0$, equation~(\ref{dbar on E}) shows that $V_{\tilde{c}}$ satisfies property {\it (ii)} of Definition~\ref{ham hol def}.
\begin{versionb}
\begin{equation*}
\dbar_{B}u_{V_{\tilde{c}}} = -\sum_{i=1}^{n} V^{i}\dbop{j}h_{i} d\bar{z}^{j}= -\frac{\sqrt{-1}}{2}\iota_{V_{\tilde{c}}}d\eta.
\end{equation*}
\end{versionb}
Conversely, suppose $V$ is a Hamiltonian, holomorphic vector field.  The expression~(\ref{local ham hol}) for $V$, along with the remarks following Definition~\ref{ham hol def} show that $V$ is in the normalizer of $\xi$ and that $[V] \in  H^{0}(\E^{1,0})$.  Computing locally we see that  the Lie bracket commutes with these identifications. Finally, by Proposition~\ref{perturb prop} any Sasaki structure on $S$ with the same complex structure $J$ on the cone, Reeb field $\xi$ and transverse holomorphic structure is related to the original Sasaki structure by deformations of type II.  One can then easily check via the above computations that $H^{0}(\E^{1,0})$ is unchanged under these deformations.
\end{proof}

\begin{corollary}\label{fut cor}
The Futaki invariant defines a character on the Lie algebra $H^{0}(E^{1,0})$.
\end{corollary}
\begin{remark}  Corollary~\ref{fut cor} is essentially a restatement of the original definition in \cite{BoyGalSim}.  However, the identification of the Lie algebra which appears in \cite{BoyGalSim} as the global holomorphic sections of the sheaf $\E$ provides what we feel to be a particularly attractive definition of the Futaki invariant, which is seen to generalize the K\"ahler setting.
\end{remark}

\begin{prop}\label{vanishing fut prop}
Let $(S,g_{0})$ be a Sasaki manifold with $c_{1}^{B}(S)>0$, and $c_{1}(D)=0$.  Suppose that $(2n+2)[\frac{1}{2}d\eta_{0}]_{B} = c_{1}^{B}(S)$.  If $\nu_{\eta_{0}}(\eta) > -\infty$ for every $\eta$ in the K\"ahler class of $\eta_{0}$, then $fut_{S}(X) \equiv 0$.
\end{prop}
\begin{proof}
Suppose $\nu_{\eta_{0}}(\eta) >-C>-\infty$, but $fut([X]) \ne 0$ for some $[X] \in H^{0}(\E^{1,0})$.  We can assume that $Re(fut([X])) <0$, by replacing $[X]$ with $[-X]$, or $[iX]$.  Let $X = \sigma_{0}([X])$.  The vector field $X$ is foliate, orthogonal $\xi$, and holomorphic.  In preferred local coordinates we have $X = \sum_{i=1}^{n} X^{i}(\pl{z^{j}}-\sqrt{-1}h_{j}\pl{x})$  for $X^{i}$ basic, holomorphic functions, and $h$ a local, real valued function.  Define the vector field $\tilde{X}$ on the cone $C(S)$, by $\tilde{X}(z,r) = Re(X)(z)$, where $z \in S$, and $r$ is the radial variable on the cone.  The local formula shows $\tilde{X}$ is real holomorphic, and so $\mathcal{L}_{\tilde{X}}J = 0$, where $J$ is the complex structure on the cone.  Let $\rho_{t}$ be the local flow of $Re(X)$ on $S$, and $\tilde{\rho}_{t}$ be the local flow of $\tilde{X}$.  Then $\tilde{\rho}_{t}$ is a biholomorphism, and it is clear that $\tilde{\rho}_{t}(z,r) = (\Phi_{t}(z), r)$.  In particular, $(S, \rho_{t}^{*}g)$ is a Sasaki manifold with the same Reeb field, the same complex structure on the cone and the same transversely holomorphic structure on the Reeb foliation.  By Proposition~\ref{perturb prop}  we have that $\rho_{t}^{*}d\eta =d\eta + \dop{B}\dbop{B}\psi_{t}$.  We can now follow the argument in \cite{TianBook} to obtain the proposition.
\end{proof}

Our primary concern will be sections of $\E$ induced from basic functions.  In order to have our theory sufficiently well adapted for our future applications we discuss this now.  Given a basic function $h$, $\dbar_{B}h$ is a section of $\Lambda^{0,1}_{B}$.  We then define $V^{j} = (g^{T})^{j\bar{k}}\dbop{k}h$.  $V$ defines a section of the quotient bundle $Q$, and the splitting map $\sigma$ satisfies $\sigma([V]) = V$.  Moreover, $V$ lies in the normalizer of $\xi$ in $TS$.  Thus, $[V]$ defines a global section of $\E^{1,0}$ over $S$.
\begin{versionb}
 given in a local preferred chart by
\begin{equation*}
[V] = \sum_{j=1}^{n}(g^{T})^{j\bar{k}}\dbop{k}h \left[\pl{z^{j}}\right].
\end{equation*} 
\end{versionb}
We now compute that
\begin{equation*}
\dop{\E}[V] =  \sum_{l=1}^{n} \nabla^{T}_{l}\left((g^{T})^{j\bar{k}}\dbop{k}h \right) \left[\pl{z^{j}}\right] \otimes dz^{l}=\sum_{l=1}^{n} (g^{T})^{j\bar{k}}\dop{l}\dbop{k}h  \left[\pl{z^{j}}\right] \otimes dz^{l},
\end{equation*}
and so
\begin{equation}\label{norm example}
\| \dop{\E}[V] \|_{L^{2}(\E^{1,0}\otimes TS^{*})} = \int_{S} (g^{T})^{l\bar{p}}(g^{T})^{j\bar{k}}\left(\nabla^{T}_{l}\nabla^{T}_{\bar{k}}h\right)\left(\overline{\nabla^{T}_{p}\nabla^{T}_{\bar{j}}h}\right)d\mu.
\end{equation}
Expressions such as these shall appear repeatedly in what is to follow.  In order to simplify our notation, we will use $\nabla$ and $\overline{\nabla}$ to denote covariant derivative in the unbarred and barred directions.  For example, equation~(\ref{norm example}) can then be written as $\| \dop{\E}[V] \|_{L^{2}} = \int_{S} |\nabla \nabla h|^{2}d\mu$.
\begin{versionb}
\begin{equation*}
\| \dop{\E}[V] \|_{L^{2}(\E^{1,0}\otimes TS^{*})} = \int_{S} |\nabla \nabla h|^{2}d\mu.
\end{equation*}
\end{versionb}

\section{Proof of Theorem~\ref{main theorem} part {\it (i)}}

In this section we use the bound below for the Mabuchi functional to show that the $L^{2}$ norm of  $\dop{B}u$ goes to zero as $t \rightarrow \infty$, where $u$ is the transverse Ricci potential, $R^{T}_{\bar{k}j} - g^{T}_{\bar{k}j} = \dop{j}\dbop{k}u$.  The uniform bounds for the transverse Riemann tensor allow us to employ an inductive argument to obtain the decay to zero of all Sobolev norms of $\dop{B}u$.  The Mabuchi K-energy along the normalized Sasaki-Ricci flow is given by~(\ref{flow mab}).
\begin{versionb}
\begin{equation*}
\nu_{\eta_{0}}(\eta(T)) = -\int_{0}^{T}\int_{S} |\dop{B} u|^{2}d\mu_{t} dt,
\end{equation*}
where $u$ is the transverse Ricci potential.  
\end{versionb}
Thus, if the Mabuchi energy is bounded below on $\mathcal{H}_{\eta_{0}}$ then there exists times $t_{k} \rightarrow \infty$ such that $\|\nabla u\|_{L^{2}}(t_{k}) \rightarrow 0$.
We can obtain convergence for the full sequence by computing the evolution equation for the quantity $Y(t) = \| u \|_{L^{2}}$.  Following the computations in~\cite{PS} we obtain
\begin{versionb}
\begin{equation*}
\pl{t} |\dop{j} u|^{2} =  (g^{T})^{j\bar{k}}\dot{(\dop{j}u)}\overline{\dop{k}u} + (g^{T})^{j\bar{k}}\dop{j}u\overline{\dot{(\dop{k}u)}} + (R^{T})^{j\bar{k}}\dop{j}\overline{\dop{k}u} -|\dop{B}u|^{2},
\end{equation*}
\begin{equation*}
\square_{B} |\dop{j}h|^{2} = (g^{T})^{j\bar{k}}(\square_{B}\dop{j}u) \overline{\dop{k}u} + (g^{T})^{j\bar{k}}\dop{j}u( \overline{\square_{B} \dop{k}u} )+ (R^{T})^{j\bar{k}}\dop{j}u\overline{\dop{k}u} + |\overline{\nabla}\nabla u|^{2} + |\nabla \nabla u|^{2}.
\end{equation*}
It follows easily that
\end{versionb}
\begin{equation}\label{Y diff equation 1}
\dot{Y(t)} = (n+1)Y(t) - \int_{S}|\dop{B}u|^{2}R^{T}d\mu - \int_{S}|\overline{\nabla}\nabla u|^{2}d\mu - \int_{S}|\nabla \nabla u|^{2}d\mu.
\end{equation}
Applying the uniform bound for $R^{T}$ in Theorem~\ref{Pman thm}, the argument from \cite{PS} carries over verbatim to yield;

\begin{Lemma}\label{PSSW lemma 4}
Assume the Mabuchi K-energy is bounded from below on the K\"ahler class of $\eta_{0}$.  
\begin{versionb}
Set
\begin{equation*}
Y(t) = \| \nabla u\|_{L^{2}}^{2}.
\end{equation*}
\end{versionb}
Then $Y(t) \rightarrow 0$ along the K\"ahler-Ricci flow as $t\rightarrow \infty$.
\end{Lemma}

\begin{proof}[Proof of Theorem~\ref{main theorem} part {\it (i)}]
In light of Lemma~\ref{PSSW lemma 4}, rearranging~(\ref{Y diff equation 1}) and integrating with respect to $t$ gives,
\begin{equation*}
\int_{0}^{\infty}dt\int_{S} |\overline{\nabla}\nabla u|^{2}d\mu_{t} + \int_{0}^{\infty}\int_{S} |\nabla \nabla u|^{2}d\mu_{t}< \infty.
\end{equation*}
We are in a position to apply our previous argument inductively.  Define $Y_{r,s}(t) = \int_{S}|\nabla^{s}\overline{\nabla}^{r}u|^{2}d\mu_{t}$.  Following the computations in \cite{PS}, and making use of the uniform $C^{\infty}$ bounds on $Rm^{T}$ and $Rm$ guaranteed by Theorem~\ref{BBS Thm}, we compute that
\begin{equation}\label{inductive inequality}
\begin{aligned}
\dot{Y}_{r,s}(t) \leq &C_{1}Y_{r,s}(t) + C_{2}\left(\int_{S}|D^{r+s-p}u|^{2}d\mu_{t}\right)^{1/2}Y_{r,s}^{1/2}(t)\\&-\int_{S}|\nabla^{s+1}\overline{\nabla}^{r}u|^{2}d\mu_{t} - \int_{S} |\overline{\nabla}\nabla^{s}\overline{\nabla}^{r}u|^{2}d\mu_{t},
\end{aligned}
\end{equation}
where summation over $1\leq p \leq r+s-1$ is understood.  We now employ the argument in \cite{PS}, which carries over verbatim.
\end{proof}

\begin{versionb}
There are two particularly attractive lines of argument which we presently discuss.  Unfortunately, both arguments succeed only in pointing out the difficulties in obtaining the existence of transverse K\"ahler-Einstein metrics when the basic first Chern class is postive, even in the presence of uniform curvature bounds.  

First,  the uniform diameter, volume and curvature bounds imply a uniform lower bound on the injectivity radius \cite{Cheeger}, and hence the Sobolev constant.  The conclusion of Theorem~\ref{main theorem} part {\it (i)}, which we have just established, combined with uniform control of the Sobolev constant might lead us to hope that the metrics $g(t)$ are converging to a transversely K\"ahler-Einstein metric.  However, the Sobolev norm in Theorem~\ref{main theorem} part {\it (i)}  is measured with respect to the evolving metrics $g(t)$, which are not known to be uniformly equivalent.  Alternatively, we can apply Corollary~\ref{flow compactness} to obtain the existence of a subsequence $t_{k} \rightarrow \infty$ and diffeomorphisms $F_{t_{k}}$ such that $(S, F_{t_{k}}^{*}g(t_{k}))$ converge in $C^{\infty}$ to a Sasaki metric $(S, g_{\infty})$.  However, we have no control over the diffeomorphisms $F_{t_{k}}$, and so we cannot conclude anything about the convergence of the original sequence $g(t_{k})$.  In the next section we will show that condition (C) gives enough control over the diffeomorphisms to conclude the exponential convergence to a Sasaki-Einstein metric.
\end{versionb}

\section{Convergence in presence of stability}

We begin this section by manipulating the equation~(\ref{Y diff equation 1}) into a more suggestive form. 
\begin{equation}\label{Y diff equation 2}
\dot{Y}(t) = -\int_{S}|\dop{B}u|^{2}(R^{T}-n)d\mu_{t} -\int_{S}\nabla^{j}u\nabla^{\bar{k}}u(R^{T}_{\bar{k}j}-g^{T}_{\bar{k}j})d\mu_{t} -2 \int_{S}|\bar{\nabla}\bar{\nabla}u|^{2}d\mu_{t}.
\end{equation}
This follows by applying the Bochner-Kodaira formula obtained in Proposition~\ref{BK formula} to the section of $\E$ defined by $V^{j} = (g^{T})^{j\bar{k}}\dbop{k}u$.  For large time $t$, the first two terms on the right hand side can easily be bounded by $\epsilon Y$, by Theorem~\ref{main theorem} part {\it (i)}.  In order to obtain the exponential decay of the quantity $Y$, we must bound the last term in equation~(\ref{Y diff equation 2}).
\begin{versionb}
\begin{equation*}
\int_{S} |\bar{\nabla}\bar{\nabla}u|^{2} d\mu_{t} > c\int_{S}|\nabla u|^{2}.
\end{equation*}
\end{versionb}
Let $\lambda_{t}$ be the smallest, strictly positive eigenvalue of the Laplacian $\square_{\E,t}$.  We include the subscript $t$ to enforce that the metric $g(t)$ is evolving.  By the elliptic theory we have
\begin{equation*}
\lambda_{t}\|V-\pi_{t}V\|_{L^{2}(\E^{1,0})} \leq \int_{S}|\dbar_{\E}V|^{2}d\mu_{t},
\end{equation*}
where $\pi_{t}$ is the $L^{2}$ projection onto $H^{0}(\E^{1,0})$ with respect to the metric $g^{T}(t)$.  As in the K\"ahler case, we observe that for the section $V \in \E$ in question, we have by Corollary~\ref{fut cor}
\begin{equation*}
\|\pi_{t}V \|^{2} = fut_{s}(\pi_{t}V).
\end{equation*}
Thus, equation~(\ref{Y diff equation 2}) yields the inequality
\begin{equation}\label{Y diff equation 3}
\begin{aligned}
\dot{Y} \leq &-2\lambda_{t}Y + 2fut_{S}\left(\pi_{t}(g^T)^{j\bar{k}}\dbop{k}u\right)\\ 
&-\int_{S}|\dop{B}u|^{2}(R^{T}-n)d\mu -\int_{S}\nabla^{j}u\nabla^{\bar{k}}u\left(R^{T}_{\bar{k}j}-g^{T}_{\bar{k}j}\right)d\mu
\end{aligned}
\end{equation}
We remark that equation~(\ref{Y diff equation 3}) is completely general, and we view it as the motivating inequality for the developments in this paper.  

\begin{proof}[Proof of Theorem~\ref{main theorem} part {\it(ii)}]
Since the Mabuchi functional is bounded below, Proposition~\ref{vanishing fut prop} implies that the Futaki invariant is zero.  The uniform transverse curvature bound, and the conclusion of Theorem~\ref{main theorem} part {\it(i)} imply that for any $\epsilon >0$, there is a $T_{\epsilon}$ such that, for every $t \in [T_{\epsilon}, \infty)$ we have $\dot{Y}(t) \leq (-\lambda_{t} +\epsilon)Y(t)$.  It suffices to find a positive lower bound for $\lambda_{t}$.  Condition (C) is tailor made for the task.

\begin{Theorem}\label{stability compactness}
Let $(S,\xi, \eta, \Phi, g)$ be a compact Sasaki manifold of dimension $2n+1$.  Assume that the Sasaki structure $(\xi, \eta, \Phi, g)$ satisfies stability condition (C).  Fix $V,D, \delta>0$, and constants $C_{k}$.  Then there exists an integer $N$ and a constant $C(V, D, \delta, C_{k}, n, N)>0$ such that
\begin{equation*}
C\|V\|^{2} \leq \|\dbar_{\E}V\|^{2}, \quad \forall \quad V \perp H^{0}(\E^{1,0})
\end{equation*}
for all Sasaki structures $(S,\xi, \eta, \Phi, g)$   with $Vol_{g}(S)<V$, and $diam_{g}(S)<D$, and whose injectivity radius is bounded below by $\delta$, and the k-th derivative of whose curvature tensors are uniformly bounded by $C_{k}$ for all $k\leq N$.
\end{Theorem}

The proof of Theorem~\ref{stability compactness} is taken up in the next section.  It follows that for $t$ sufficiently large, we have $Y(t) \leq Ce^{-ct}$.  With the exponential decay of the $L^{2}$ norm of $|\nabla u|$ established, a straight forward adaptation of the arguments in \cite{PS} yield the exponential decay of the $L^{2}$ norms of $\bar{\nabla}^{r}\nabla^{s}u$ where all norms are computed with respect to the evolving metric $g^{T}(t)$. The Sobolev imbedding theorem with uniform constants then gives the exponential decay of the $C^{k}$ norm of $u$ for any $k$.  Since $\dot{g}^{T}_{\bar{k}j} = \dop{j}\dbop{k}u$, we have $\int_{0}^{\infty}\sup_{S}|\dot{g}^{T}_{\bar{k}j}|_{g^{T}(t)}dt <\infty$.  A lemma of Hamilton (\cite{RHam1} Lemma 14.2) allows us to conclude that the metrics $g^{T}_{\bar{k}j}$ on $Q$ are uniformly equivalent. While Hamilton's proof is for metric tensors on the tangent bundle,  one can easily check that the argument holds for vector bundles.
\begin{versionb}
\begin{Lemma}[\cite{RHam1} Lemma 14.2]\label{uniform eq lemma}
Let $g_{ij}(t)$ be a time dependent metric on a vector bundle $E \rightarrow S$ for $0\leq t<T\leq \infty$. Suppose
\begin{equation*}
\int_{0}^{T} \max_{S}|\dot{g}_{ij}|dt = \int_{0}^{T}\max_{S} \sqrt{g^{ik}g^{jl}\dot{g}_{ij}\dot{g}_{kl}}dt \leq C < \infty.
\end{equation*}
Then the metrics $g_{ij}(t)$ for all different times are uniformly equivalent, in the sense that there exists a constant $C'$ independent of $t$ such that
\begin{equation*}
C'^{-1}g(t) \leq g(0)\leq C'g(t).
\end{equation*}
Moreover, the metrics $g_{ij}$ converge as $t\rightarrow T$ uniformly to a positive-definite metric $g_{ij}(T)$ which is continuous and also equivalent.
\end{Lemma}
\begin{remark}
Hamilton's proof is for metric tensors on the tangent bundle.  One can easily check that the argument yields the above, more general result.
\end{remark}
\end{versionb}
The uniform equivalence of the metrics $g^{T}$ imply that for any section $W \in Q$ we have
\begin{equation*}
|g^{T}_{\bar{k}j}(\mathcal{T})W^{j}\overline{W^{k}} - g^{T}_{\bar{k}j}(\mathcal{S})W^{j}\overline{W^{k}}|\leq C|W|^{2}_{t=0}\left(e^{-c\mathcal{T}}-e^{-c\mathcal{S}}\right)
\end{equation*}
\begin{versionb}
\begin{equation*}
\begin{aligned}
&|g^{T}_{\bar{k}j}(\mathcal{T})W^{j}\overline{W^{k}} - g^{T}_{\bar{k}j}(\mathcal{S})W^{j}\overline{W^{k}}|\\
&\leq \int_{\mathcal{S}}^{\mathcal{T}}|\dot{g}^{T}_{\bar{k}j}||W|^{2}_{t}dt\\
&\leq C|W|^{2}_{t=0}\left(e^{-c\mathcal{T}}-e^{-c\mathcal{S}}\right).
\end{aligned}
\end{equation*}
\end{versionb}
As the last term goes to zero exponentially as $\mathcal{S}, \mathcal{T} \rightarrow \infty$, we obtain the exponential convergence of the transverse metrics to some metric $g^{T}_{\bar{k}j}$ which is equivalent to all the metrics $g^{T}_{\bar{k}j}(t)$.  Iteration yields exponential convergence in $C^{\infty}$.  Since $\dop{j}\dbop{k}h$ tends to zero, the limiting metric $g^{T}_{\bar{k}j}(\infty)$ is Sasaki-Einstein.
\end{proof}

\section{Compactness Theorems and the Proof of Theorem~\ref{stability compactness}}
Our main objective in this section is to prove Theorem~\ref{stability compactness}, which will finish the proof of Theorem~\ref{main theorem}.  We begin by stating and proving a Sasaki version of Gromov compactness.  This theorem is well known, and follows easily from  Hamilton's compactness theorem \cite{RHam} but we include the short proof for completeness. 

\begin{Theorem}\label{Ham compact}
Let $(S,g)$ be a compact Sasaki manifold.  Let $g(t)$ be a sequence of Sasaki metrics on $S$, and $J(t)$ a sequence of complex structures on the cone $C(S)$ such that $(C(S),\bar{g}(t):= dr^{2}+r^{2}g(t), J(t))$ is K\"ahler.  Assume that the $g(t)$'s have bounded geometry in the sense that their volumes, diameters, curvatures and covariant derivatives of their curvature tensor are all bounded from above, and their injectivity radii are bounded below.  Then there exists a subsequence $\{t_{j}\}$, and a sequence of diffeomorphisms $F_{t_{j}}:S \rightarrow S$ such that the pulled back metrics $F_{t_{j}}^{*}g(t_{j})$ converge in $C^{\infty}$ to a smooth metric $\tilde{g}(\infty)$.  Moreover, on the cone, the lifted diffeomorphisms defined by $\tilde{F}_{t_{j}} (r,z) := (r, F_{t_{j}}(z))$ have that  the sequence $\tilde{F}_{t_{j}}^{*}J(t_{j}))$ converges in $C^{\infty}$ to an integrable complex structure $\tilde{J}(\infty)$ on $C(S)$.  Furthermore, the metric $\tilde{g}(\infty)$ is Sasaki with respect to the complex structure $\tilde{J}(\infty)$.  In particular, the Sasaki structures $(\xi(t_{j}), \eta(t_{j}), \Phi(t_{j}), g(t_{j}))$ converge in $C^{\infty}$ to a Sasaki structure $(\tilde{\xi}, \tilde{\eta}, \tilde{\Phi}, \tilde{g})$. 
\end{Theorem}

\begin{proof}
The $C^{\infty}$ convergence part of this theorem is just Hamilton's compactness theorem \cite{RHam}.  Thus, we are reduced to showing that the complex structures converge.  This follows essentially from the proof of \cite{PS} Theorem 4, with the wrinkle that the cone $C(S)$ is not compact.  Consider instead the truncated cone $\tilde{C}(S) = ((\frac{1}{2},1)\times S)$. The argument in \cite{PS} shows that the complex structures converge to an integrable complex structure $\tilde{J}(\infty)$ on compact sets making $(\tilde{C}(S), \tilde{J}(\infty), dr^{2} + r^{2}\tilde{g}({\infty}))$ into a K\"ahler manifold.  We extend the complex structure to the whole cone by using the fact that $\mathcal{L}_{r\dop{r}}J(t_{j}) =0$, and that $\tilde{F}_{t_{j}*}r\dop{r} = r\dop{r}$.
\end{proof}

\begin{proof}[Proof of  Theorem~\ref{stability compactness}.]
Let $\lambda_{t}$ be the smallest positive eigenvalue of $\square_{\E}$, defined with respect to the Sasaki structure $\mathfrak{s}(t):= (\xi(t), \eta(t), \Phi(t), g(t))$.  Suppose $\mathfrak{s}(t)$ converges in $C^{\infty}$ to a Sasaki structure $\mathfrak{s}_{\infty} :=(\xi_{\infty}, \eta_{\infty}, \Phi_{\infty}, g_{\infty})$ and the dimension of the space of global holomorphic sections of the sheaf of transverse foliate vector fields is the same for every $N \leq t \leq \infty$.  The perturbation theory for the Laplacian used in the proof of Theorem 3 extends to the global sections of the sheaf $\E$, and we obtain
\begin{equation}\label{limit eigenvalues}
\lim_{t\rightarrow \infty} \lambda_{t} = \lambda_{\infty}.
\end{equation}
We can now prove by contradiction:  assume there exists a sequence of metrics $g(t)$ with $\lambda_{t} \rightarrow 0$.  By passing to a subsequence (not relabeled), we can apply Theorem~\ref{Ham compact} to obtain the existence of diffeomorphisms $F_{t}$ so that the pulled back Sasaki structure 
\begin{equation*}
\tilde{\mathfrak{s}}(t):=(F_{t}^{*}g(t), F_{t}^{*}\xi(t), F_{t}^{*}\eta(t),  F_{t}^{*}\Phi(t))
\end{equation*}
converges in $C^{\infty}$ to a Sasaki structure $\tilde{\mathfrak{s}}_{\infty}=(\tilde{\xi}_{\infty}, \tilde{\eta}_{\infty}, \tilde{\Phi}_{\infty}, \tilde{g}_{\infty})$.  By equation~(\ref{limit eigenvalues}), the lowest positive eigenvalue of $\tilde{\mathfrak{s}}(t)$ converges to a strictly positive limit.  Let $\E(t), \tilde{\E}(t)$ be the sheaves defined by the Sasaki structures $\mathfrak{s}(t), \tilde{\mathfrak{s}}(t)$ respectively.  Observe that the diffeomorphism $F_{t}$ induces an isomorphism of sheaves $\E(t) \cong \tilde{\E}(t)$.  Moreover, $F_{t}$ is an isometry which preserves the transverse holomorphic structure, and hence descends to an isometry of the quotient bundles $Q(t), \tilde{Q}(t)$.  Using the computations in \S5, it is then clear that the sheaves $\E(t), \tilde{\E}(t)$ are isospectral, providing a contradiction.

\end{proof}


\section{The Proof of Theorem~\ref{main theorem 2}}

Note that in the proof of Theorem~\ref{main theorem} we only needed a bound on the smallest positive eigenvalue of $\square_{\E}$ restricted to sections of $\E^{1,0}$ induced by basic functions.  Rather than study the $\dbar_{\E}$ Laplacian on global sections of $\E^{1,0}$ we are thus motivated to study the following operator on $C^{\infty}_{B}$;
\begin{equation*}
L := -(g^{T})^{j\bar{k}}\nabla_{j}\nabla_{\bar{k}} + (g^{T})^{j\bar{k}}\nabla_{j}u\nabla_{\bar{k}}.
\end{equation*}
Our developments will require the basic Sobolev and Lebesgue spaces, for which we refer the reader to \cite{TristC1}.  The operator $L$ appeared in \cite{TristC}, where it was shown to be elliptic and self-adjoint with respect to the $L^{2}$ inner-product on $H^{2}_{B}$ induced by the probability measure $d\rho$ defined in \S3.  Moreover, it was shown that $L$ has a complete spectrum of smooth, basic eigenfunctions $\{ \psi_{j} \}_{j \in \mathbb{N}}$ spanning $L^{2}_{B}$, with eigenvalues $\lambda_{j} \geq 1$.  This yields the Poincar\'e inequality;

\begin{Lemma}\label{poincare}
Let $u$ satisfy the equation $g^{T}_{\bar{k}j}-Ric^{T}_{\bar{k}j} = \partial_{j}\partial_{\bar{k}}u$.  Then the following inequality
\begin{equation*}
\frac{1}{Vol(S)}\int_{S}f^{2}e^{-u}d\mu \leq \frac{1}{Vol(S)}\int_{S}|\nabla f|^{2}e^{-u}d\mu +\left(\frac{1}{Vol(S)}\int_{S}f e^{-u} d\mu\right)^{2}
\end{equation*}
holds for all $f\in C^{\infty}_{B}(S)$.
\end{Lemma}

Observe that if $\lambda_{t} \geq 1+\delta >1$, then a standard computation for the operator $L$ suggests that the final term in~(\ref{Y diff equation 2}) can be controlled by $-\delta Y(t)$ (cf. equation (29) in \cite{TristC}).  However, as we are not assuming a lower bound for the Mabuchi energy, it is no longer natural to work with $Y(t)$.  Instead, we define
\begin{equation*}
W(t):= \frac{1}{Vol(S)}\int_{S} (u-a)^{2} e^{-u} d\mu, \quad Z(t) := \pl{t}a = \frac{1}{Vol(S)}\int_{S} \left( |\nabla u |^{2} - (u-a)^{2} \right) e^{-u}d\mu,
\end{equation*} 
where $a(t)$ is defined by~(\ref{a def}).  The main result in this section is the following proposition.

\begin{prop}\label{PSSW lemma 5}
Assume that condition (F) holds on $(n+1)^{-1}c_{1}^{B}(X)$, and that condition (T) holds along the Sasaki-Ricci flow with initial value $g_{0}$.  Then there are constants $b, C >0$ independent of $t$ so that $W(t) \leq Ce^{-b t}$, for every $ t\in [0,\infty)$.  Moreover,
\begin{equation*}
\|u\|_{C^{0}} + \| \nabla u \|_{C^{0}} + \|R^{T}-(2n+2) n\|_{C^{0}} \leq Ce^{-\frac{1}{2(n+1)}b t}, \quad t\in [0,\infty).
\end{equation*}
\end{prop}
The proof of Proposition~\ref{PSSW lemma 5} proceeds in several steps.  First, we describe a smoothing lemma which will reduce the proof of the proposition to proving the exponential decay of $W$.  The important idea of a smoothing lemma was first introduced by Bando \cite{Bando}, and appeared in \cite{PSSW}; it was subsequently improved in \cite{Dono}.

\begin{Lemma}\label{smoothing lemma}
There exist positive constants $\delta, K$ depending only on $n$ with the following property; for any $\epsilon \in (0,\delta]$, and any $t_{0} >0$, if $\|u(t_{0})\|_{C^{0}} \leq \epsilon$, then
\begin{equation*}
\|\nabla u(t_{0} +2)\|_{C^{0}} + \|R(t_{0}+2) -(2n+2)n\|_{C^{0}} \leq K\epsilon.
\end{equation*}
\end{Lemma}
The proof of Lemma~\ref{smoothing lemma} is identical to the K\"ahler case, and can be found in \cite{PSSW}. In order to prove Proposition~\ref{PSSW lemma 5}, we must first establish the exponential decay of $W$.  We now describe a condition under which such decay holds.

\begin{prop}\label{Zhang  thm 1.5}
Suppose there is a uniform constant $\delta >0$ independent of time so that
\begin{equation}\label{Zhang  thm 1.5 ass}
\frac{(1+\delta)}{V}\int_{S}\left(u(t)- a(t)\right)^{2}e^{-u(t)}d\mu_{t} \leq \frac{1}{V}\int_{S}|\nabla u(t)|^{2}e^{-u(t)}d\mu_{t}. 
\end{equation}
Then there are constants $b, C >0$ independent of $t$ so that $W(t) \leq Ce^{-b t}$.
\end{prop}
The proof of this proposition requires the following lemma, which is a consequence of the developments in \cite{TristC}.
\begin{Lemma}\label{PSSW lemma 3}
The transverse Ricci potential $u(t)$, and its average $a(t)$ satisfy the following inequalities, where the constants $C_{1}$ and $C_{2}$ depend only on $g_{0}$.
\begin{enumerate}
\item[{\it (i)}] $0 \leq -a \leq \| u-a\|_{C^{0}}$
\item[{\it (ii)}] $\|u-a\|^{n+1}_{C^{0}} \leq C_{1}\| \nabla u \|^{2}_{C^{0}} \|u-a \|_{L^{2}} \leq  C_{2} \|\nabla u\|_{L^{2}} \|\nabla u\|_{C^{0}}^{n}$
\end{enumerate}
\end{Lemma}
\begin{proof}
We combine Lemma~\ref{poincare}, Proposition~\ref{non-collapse 1} and follow the proof in \cite{PSSW}.
\end{proof} 

\begin{proof}[Proof of Proposition~\ref{Zhang thm 1.5}]
The proof follows essentially from the arguments in \cite{Zhang}, and so we provide only a sketch.  We first claim that $W(t) \rightarrow 0$ as $t\rightarrow \infty$.  To see this, observe that $Z(t) \geq \delta W(t)$ by assumption.  From Theorem~\ref{Pman thm} we have
\begin{equation*}
\int_{0}^{\infty}Z(t)dt  = \lim_{t\rightarrow \infty}a(t) - a(0) < \infty.
\end{equation*}
Thus, $Z(t) \rightarrow 0$ along a subsequence $t_{k} \rightarrow \infty$.  The convergence of the full sequence is obtained as in \S6, by computing the evolution equation for $W$.  Lemma~\ref{PSSW lemma 3}, combined with the uniform bounds in Theorem~\ref{Pman thm} imply that $\|u-a\|_{C^{0}} \leq A W(t)^{1/(2n+2)}$, and so $u \rightarrow a$ in $C^{0}$ as $t$ goes to infinity.  The result follows from elementary modifications to the proof of Lemma 2.4 in \cite{Zhang}.
\end{proof}

Lemma~\ref{PSSW lemma 3} and the uniform bounds for $u$ in Theorem~\ref{Pman thm} imply that if $W(t)$ decays exponentially, then $\|u \|_{C^{0}}$ decays exponentially.  By Lemma~\ref{smoothing lemma}, we see that the second statement in Proposition~\ref{PSSW lemma 5} follows from the exponential decay of $W(t)$.  Our task is now reduced to showing that when conditions (T) and (F) hold, the assumptions of Proposition~\ref{Zhang thm 1.5} are satisfied.  We begin by showing that if the Futaki invariant vanishes and we have a non-degeneracy condition on the `second' eigenvalue of $L$, then~(\ref{Zhang  thm 1.5 ass}) holds.

\begin{prop}\label{Zhang thm 1.5}
Let $\nu(t)$ be the smallest eigenvalue of $L$ larger than one.  Assume that $\nu(t) \geq 1+\delta$ for some $\delta>0$ uniformly along the flow.  If the Futaki invariant vanishes, then~(\ref{Zhang  thm 1.5 ass}) holds.
\end{prop}
\begin{proof}
Fix a time $t$, and from now on suppress the $t$ variable.  Recall that in \S5 it was pointed out the $(g^{T})^{i\bar{j}}\dbop{j}u$ defines a section of $\E^{1,0}$.  For simplicity we denote this section by $\nabla u \in \E^{1,0}$.  Since $fut_{S} \equiv 0$, we necessarily have $\nabla u \perp H^{0}(\E^{1,0})$ in $L^{2}(\E^{1,0}, d\mu)$.  Let $\pi$ denote the orthogonal projection  to $H^{0}(\E^{1,0})$ in the space in $L^{2}(\E^{1,0}, d\rho)$.  We decompose $\nabla u = \pi(\nabla u) + V$, then we have
\begin{equation}\label{Z lemma 3.2}
\langle \pi(\nabla u), \pi(\nabla u) \rangle_{L^{2}(\E^{1,0}, d\mu)} \leq \langle V, V \rangle_{L^{2}(\E^{1,0}, d\mu)}.
\end{equation}
Let $\nu_{0} =1 < \nu_{1} <\nu_{2} < \dots $ be the distinct eigenvalues of $L$ acting on the function space $H^{2}_{B}(d\rho)$, and let $E_{k}$ be the eigenspace of $\nu_{k}$.  $E_{0}$ and may be empty, and corresponds to those basic functions which induce sections of $H^{0}(\E^{1,0})$.  Let $u-a = u_{0} + u_{1} +\dots$ be the unique decomposition of $u-a$ into eigenfunctions $u_{k} \in E_{k}$ so that $u_{i}$ and $u_{j}$ are orthogonal in $L^{2}(d\rho)$ for $i\ne j$.  For $k \geq 1$ we have $\nu _{k} \geq \nu_{1} >1+\delta$  uniformly along the flow, and so
\begin{equation*}
\begin{aligned}
\int_{S}(u-a)^{2}d\rho &= \sum_{k} \int_{S} |u_{k}|^{2}d\rho = \sum_{k} \lambda_{k}^{-1}\int_{S}|\nabla u_{k}|^{2}d\rho \\
&\leq \int_{S} |\nabla u_{0}|^{2} d\rho + \sum_{k=1}^{\infty}(1+\delta)^{-1}\int_{S}|\nabla u_{k}|^{2}\\
&= \int_{S} |\nabla u_{0}|^{2} +(1+\delta)^{-1} \int_{S}|V|^{2}d\rho
\end{aligned}
\end{equation*}
We claim that $\pi(\nabla u) = \nabla u_{0}$ as sections of $\E^{1,0}$.  Assuming this for the moment, we obtain from~(\ref{Z lemma 3.2})
\begin{equation*}
\int_{S}|\nabla u_{0} |^{2}d\rho \leq \frac{e^{-\inf u}}{Vol(S)} \langle V, V \rangle_{L^{2}(\E^{1,0}, d\mu)}\leq e^{osc(u)}\int_{S}|V|^{2}d\rho.
\end{equation*}
It follows that
\begin{equation*}
\int_{S} (u-a)^{2}d\rho \leq \int_{S} |\nabla u|^{2} d\rho - \frac{\delta}{1+\delta} \int_{S}| V |^{2}d\rho \leq \left(1- \frac{\delta}{(1+\delta)(1+e^{osc(u)})}\right) \int_{S}|\nabla u|^{2}d\rho.
\end{equation*}
From this, one easily shows that equation~(\ref{Zhang  thm 1.5 ass}) holds.  It suffices to establish the claim.  In fact, we shall prove something more general.  Suppose that $\psi$ is an eigenfunction of $L$ with eigenvalue $\nu >1$.  By elliptic regularity, $\psi \in C^{\infty}_{B}$.  Denote by $\nabla \psi$ the global section of $\E^{1,0}$ induced by $\dbar_{\E}\psi$; we claim that $\nabla \psi \perp H^{0}(\E^{1,0})$ in $L^{2}_{B}(d\rho)$.  To see this, let $V \in H^{0}(\E^{1,0})$, and compute
\begin{equation*}
\begin{aligned}
\langle V, \nabla \psi \rangle_{L^{2}(d\rho)} &= \int_{S}V^{i}\dop{i}\bar{\psi} e^{-u}d\mu \\
& = \nu^{-1}\int_{S}V^{i}(g^{T})^{l\bar{k}}\left(-\nabla^{T}_{i}\dop{l}\dbop{k}\bar{\psi} +\dop{i}\dbop{k}u\dop{l}\bar{\psi} + \dbop{k}u\nabla^{T}_{i}\dop{l}\bar{\psi} \right) e^{-u} d\mu \\
&=\nu^{-1}\int_{S}V^{i}(g^{T})^{l\bar{k}}\left(-\nabla^{T}_{l}\nabla^{T}_{\bar{k}}\dop{i}\bar{\psi} +(-Ric^{T}_{\bar{k}i}+g^{T}_{\bar{k}i})\dop{l}\bar{\psi} + \dbop{k}u\nabla^{T}_{i}\dop{l}\bar{\psi} \right) e^{-u} d\mu\\
&=\nu^{-1}\int_{S}V^{i}\dop{i}\bar{\psi}d\rho,
\end{aligned}
\end{equation*}
where the last line follows by commuting covariant derivatives and integrating by parts.  Since $\nu >1$, we obtain the claim.
\end{proof}


\begin{proof}[Proof of Proposition~\ref{PSSW lemma 5}]
Let $\nu$ denote the first eigenvalue of $L$ strictly larger than 1.  In light of Proposition~\ref{Zhang thm 1.5}, it suffices to show that when condition (T) holds there is a $\delta >0$ such that $\nu >1+\delta$ along the flow. Let $\lambda, \tilde{\lambda}$ be the smallest positive eigenvalues of the $\dbar_{\E}$ operator acting on $L^{2}(\E^{1,0}, d\mu)$ and $L^{2}(\E^{1,0}, d\rho)$ respectively.  That is, we have
\begin{equation*}
\int_{S} |\overline{\nabla}V|^{2} d\mu \geq \lambda \int_{S} |V|^{2} d\mu, \quad \text{ for all } V \perp H^{0}(\E^{1,0}) \text{ in } L^{2}(d\mu)
\end{equation*} 
\begin{equation*}
\int_{S} |\overline{\nabla}V|^{2} d\rho \geq \tilde{\lambda} \int_{S} |V|^{2} d\rho, \quad \text{ for all } V \perp H^{0}(\E^{1,0}) \text{ in } L^{2}(d\rho)
\end{equation*}
One can easily check that $e^{- osc(u)}\lambda \leq \tilde{\lambda} \leq e^{osc(u)} \lambda$, see for example \cite{PSSW1, Zhang}.  We now show that $\nu > 1+e^{-osc(u)}\lambda$, which suffices to establish Proposition~\ref{PSSW lemma 5}.  Let $\psi \in L^{2}_{B}$ be an eigenfunction of $L$ with eigenvalue $\lambda$, and let $\nabla \psi$ denote the section of $\E^{1,0}$ induced by $\dbar_{\E}\psi$.  From the proof of Proposition~\ref{Zhang thm 1.5}, we know that $\nabla \psi \perp H^{0}(\E^{1,0})$ in $L^{2}_{B}(d\rho)$.  We have (cf. equation (29) in \cite{TristC})
\begin{equation*}
(\nu -1) \int_{S}|\nabla \psi|^{2}d\rho = \int_{S}|\dbar_{\E}\nabla \psi|^{2} d\rho \geq \tilde{\lambda}\int|\nabla \psi |^{2}d\rho.
\end{equation*}
It follows that $\nu > 1+e^{-osc(u)}\lambda$.  By the uniform $C^{0}$ bound for $u$ in Theorem~\ref{Pman thm} we see that if condition (T) holds, then there is $\delta >0$ such that $\nu >1+\delta$ uniformly along the flow.
\end{proof}

The following general lemma gives a condition under which the Sasaki-Ricci flow converges; in light of Proposition~\ref{PSSW lemma 5}, it finishes the proof of Theorem~\ref{main theorem 2}.

\begin{versionb}
\begin{proof}
Note that the second statement follows from the first and Lemma~\ref{smoothing lemma}, and so it suffices to establish the exponential decay of $Y(t)$.  Our starting point is equation~(\ref{Y diff equation 3}).  Since the Mabuchi K-energy is bounded below, Proposition~\ref{vanishing fut prop} implies the Futaki invariant vanishes identically.  Let $\lambda >0$ satisfy $\lambda_{t} \geq \lambda>0$, as guaranteed by condition (T).  It is enough to obtain a bound for the last term on the right hand side of equation~(\ref{Y diff equation 3}).
\begin{equation*}
\left|\int_{S} \nabla^{j}u\nabla^{\bar{k}}u (R^{T}_{\bar{k}j} -g^{T}_{\bar{k}j})d\mu\right|\leq \frac{\lambda}{2} Y.
\end{equation*}
This would follow, for example, from Theorem~\ref{main theorem} part {\it (i)} if we knew that $Rm^{T}$ was uniformly bounded along the flow.  In the absence of this information, we must work harder.  
By the Cauchy-Schwartz, and integration by parts we have
\begin{equation*}
\left|\int_{S} \nabla^{j}u\nabla^{\bar{k}}u (R^{T}_{\bar{k}j} -g^{T}_{\bar{k}j})d\mu\right| \leq Y(t)^{1/2}\left( \int_{S} |R^{T}-(2n+2)n|^{2} d\mu\right)^{1/2}.
\end{equation*}
\begin{versionb}
\begin{equation*}
\left|\int_{S} \nabla^{j}u\nabla^{\bar{k}}u (R^{T}_{\bar{k}j} -g^{T}_{\bar{k}j})d\mu\right| \leq Y(t)^{1/2}\left(\int_{S}\left| R^{T}_{\bar{k}j} -g^{T}_{\bar{k}j}\right|^{2}d\mu\right)^{1/2}
\end{equation*}
By the definition of the Ricci potential and integration by parts we have
\begin{equation*}
\int_{S}\left| R^{T}_{\bar{k}j} -g^{T}_{\bar{k}j}\right|^{2}d\mu =  \int_{S} |R^{T}-(2n+2)n|^{2} d\mu.
\end{equation*}
\begin{equation*}
\begin{aligned}
\int_{S}\left| R^{T}_{\bar{k}j} -g^{T}_{\bar{k}j}\right|^{2}d\mu &= \int_{S}|\dop{j}\dbop{k} u|^{2}d\mu\\
&= \int_{S} |\square_{B} u|^{2}d\mu\\
&= \int_{S} |R^{T}-(2n+2)n|^{2} d\mu.
\end{aligned}
\end{equation*} 
We can now apply Lemma~\ref{smoothing lemma} and Lemma~\ref{PSSW lemma 3} to obtain the desired bound.  By Theorem~\ref{main theorem 3} part {\it(ii)}, there is a $T \in [0, \infty)$ such that , for all $t \geq T$, $|u(t)|_{C^{0}} \leq \delta$, where $\delta$ is as in the statement of Lemma~\ref{smoothing lemma}.  From now on, we assume $t \gg 2T $.  Applying Lemma~\ref{smoothing lemma} and Lemma~\ref{PSSW lemma 3}, we obtain
\begin{equation*}
|R^{T}-(2n+2)n| \leq C \|\nabla u(t-2) \|_{C^{0}}^{\frac{n}{n+1}} \| \cdot \| \nabla u(t-2)\|_{L^{2}}^{\frac{1}{n+1}},
\end{equation*} 
for $C$ depending only on $g_{0}$ and $n$.  We can now apply Lemma~\ref{smoothing lemma} to the term involving $\| \nabla u \|_{C^{0}}$, and iterate this argument.  In particular, suppose that $G(t)$ is a function of the form
\begin{equation*}
G(t) = \displaystyle\prod_{j} A(t-a_{j})^{\delta_{j}} \cdot  \displaystyle\prod_{k} B(t-b_{k})^{\epsilon_{k}},
\end{equation*}
where $A(t) = \| \nabla u|_{L^{2}}(t)$, $B(t) = \| \nabla u \|_{C^{0}}(t)$, $a_{j}, b_{k}, \delta_{j}, \epsilon_{k}$ are positive numbers and $\delta := \sum_{j} \delta_{j}$, and $\epsilon:= \sum_{k} \epsilon_{k}$ satisfy $\delta +\epsilon =1$.  Applying Lemmas~\ref{smoothing lemma} and~\ref{PSSW lemma 3} we obtain that $G(t) \leq C \tilde{G}(t)$ where $\tilde{G}(t) = \prod_{j}A(t-\tilde{a}_{j})^{\tilde{\delta}_{j}} \cdot \prod_{k} B(t-\tilde{b}_{k})^{\tilde{\epsilon}_{k}}$, and $\tilde{\delta} + \tilde{\epsilon} =1$, but now $\tilde{\epsilon} = n\epsilon/(n+1)$.  Using that $\| u \|_{C^{0}} \rightarrow 0$ as $t \rightarrow \infty$ we obtain
\begin{equation*}
\dot{Y}(t) \leq -\lambda Y(t) + \frac{\lambda}{2}Y^{\frac{1}{2}}(t)\cdot \Pi_{j=1}^{N}\left[ Y(t-a_{j})\right]^{\delta_{j}/2}
\end{equation*}
for all $t \geq K_{0}$, where $N$ is an integer and the $\delta_{j}$ are non-negative real numbers with the property that $\sum_{j=1}^{N} \delta_{j} =1$.  The arguments of \cite{PSSW} apply verbatim, and so we omit the details.
\end{proof}

The techniques used in the proof of Lemma~\ref{PSSW lemma 5} give a proof of Theorem~\ref{main theorem} part {\it (i)}.  We delay the argument for now, so that we may discuss the proof of Theorem~\ref{main theorem 2}.
\end{versionb}
\begin{Lemma}\label{PSSW lemma 6}
Assume that the transverse scalar curvature $R^{T}(t)$ along the Sasaki-Ricci flow satisfies
\begin{equation*}
\int_{0}^{\infty} \|R^{T}(t) - (2n+2) n\|_{C^{0}} dt < \infty.
\end{equation*}
Then the metrics $g(t)$ converge exponentially fast to a Sasaki-Einstein metric.
\end{Lemma}
\begin{proof}
The Sasaki potential $\phi(t)$, satisfies equation~(\ref{PCMA}), with $F = u(0)$, and $\phi(0)=c_{0}$,
\begin{versionb}
\begin{equation}
\dot{\phi} = \log \frac{(\frac{1}{2}d\eta)^{n}\wedge \eta_{0}}{ (\frac{1}{2}d\eta_{0})^{n}\wedge \eta_{0}} +(2n+2)\phi - u(0), \quad \phi(0)=c_{0},
\end{equation}
\end{versionb}
where $c_{0}$ is defined by~(\ref{initial condition}).  For this particular choice of initial condition, Theorem~\ref{Pman thm} implies that $\| \dot{\phi}\|_{C^{0}} \leq C$ uniformly along the flow.  Now,
\begin{equation*}
\pl{t}  \log \frac{(\frac{1}{2}d\eta)^{n}\wedge \eta_{0}}{ (\frac{1}{2}d\eta_{0})^{n}\wedge \eta_{0}} = -(R^{T}-(2n+2) n),
\end{equation*}
and so our assumption implies
\begin{equation*}
\left| \log \frac{(\frac{1}{2}d\eta)^{n}\wedge \eta_{0}}{ (\frac{1}{2}d\eta_{0})^{n}\wedge \eta_{0}}\right| = \left| \int_{0}^{t} (R^{T}-(2n+2) n) dt\right| \leq \int_{0}^{\infty} \|R^{T}-(2n+2) n\|_{C^{0}}dt < \infty.
\end{equation*}
Rearranging equation~(\ref{PCMA}) as an equation for $\phi$, and using the uniform bound for $\dot{\phi}$ we obtain that $\phi$ is uniformly bounded.
\begin{versionb}
\begin{equation*}
(2n+2) \phi = \dot{\phi} -  \log \frac{(\frac{1}{2}d\eta)^{n}\wedge \eta_{0}}{ (\frac{1}{2}d\eta_{0})^{n}\wedge \eta_{0}}  + u(0).
\end{equation*}
The right hand side is uniformly bounded, and so we obtain a uniform bound for $\phi$. 
\end{versionb}
By Proposition~\ref{uniform Yau}, $\| \phi \|_{C^{k}}$ is uniformly bounded for each $k \in \mathbb{N}$, where the $C^{k}$ norm is with respect to the initial metric $g^{T}(0)$.  The uniform bounds on $\phi$ imply that the metrics $g^{T}(t)$ are uniformly equivalent and uniformly  bounded in $C^{\infty}$; in particular, $Rm^{T}$ is uniformly bounded.  It follows that there exists a subsequence of times $t_{m} \rightarrow \infty$ with $\phi(t_{m})$ converging in $C^{\infty}$ to smooth basic function $\phi(\infty)$.  By uniform equivalence we have
\begin{equation*}
\left| \square_{B, g(0)} u(t_{m}) \right| \leq C \left| \square_{B, g(t)} u(t_{m})\right| \leq C |R^{T}(t_{m})-(2n+2) n|_{C^{0}} \longrightarrow 0.
\end{equation*}
Thus, $\phi(\infty)$ is a potential for a transversely K\"ahler-Einstein metric.  Let $\lambda_{t}$ be the smallest positive eigenvalue of $\square_{\E}$ acting on smooth global sections of $\E^{1,0}$.  We claim that $\lambda_{t} \geq \lambda >0$.  If this were not the case, then there is a further subsequence (not relabeled) such that $g(t_{m})$ converges in $C^{\infty}$ to a Sasaki metric $\tilde{g}$, and $\lambda_{t_{m}}\rightarrow 0$.  We can now apply the arguments in the proof of Theorem~\ref{stability compactness} in the special case that the Reeb field $\xi$ and the transverse complex structure $\Phi$ are fixed, and $\eta(t) = \eta_{0} +2d_{B}^{c}\phi(t)$.  In particular, by Proposition~\ref{kernel prop} the dimension of the space of global holomorphic sections of $\E^{1,0}$ is constant.  We then obtain $0 = \lim_{m\rightarrow \infty} \lambda_{t_{m}} = \lambda (\tilde{g})>0$, which is a contradiction.  Proposition~\ref{Mab prop} implies that the Mabuchi K-energy is bounded below and so by Lemma~\ref{PSSW lemma 5} we obtain the exponential decay to zero of $Y(t) = \| \nabla u \|_{L^{2}}^{2}$.  We claim that this implies the exponential decay to zero of $\|\nabla u\|_{(s)}$ for any Sobolev norm $\| \cdot \|_{(s)}$.  This follows essentially from the computations in \S6.  For example, rearranging equation~(\ref{Y diff equation 1}), we get
\begin{equation}\label{sobolev norm decay}
 (n+1)Y(t) - \int_{S}|\dop{B}u|^{2}R^{T}d\mu-\dot{Y(t)}=  \int_{S}|\overline{\nabla}\nabla u|^{2}d\mu + \int_{S}|\nabla \nabla u|^{2}d\mu.
 \end{equation}
Thus, the uniform bound for $R^{T}$ and the exponential decay of $Y$ yields the exponential decay of the right hand side of equation~(\ref{sobolev norm decay}).  We then proceed inductively, using the functions $Y_{r,s}(t)$ as defined in \S6 and and employing the aforementioned uniform curvature bounds (cf. equation~(\ref{inductive inequality})).   Since the transverse Ricci potential $u$ is basic, and the metrics $g^{T}(t)$ are uniformly equivalent,  the Sobolev imbedding theorem yields the exponential decay to zero of $\| u\|_{C^{k}}$ for any $k$, and hence $\| \dot{g}^{T}_{\bar{k}j} \|_{C^{k}} = \|R^{T}_{\bar{k}j}-(2n+2)g^{T}_{\bar{k}j}\|_{C^{k}}$ decays exponentially to zero for any $k$.
\end{proof}

\begin{proof}[Proof of Theorem~\ref{main theorem 2}]
Part {\it (i)} follows from Proposition~\ref{PSSW lemma 5}, and Lemma~\ref{PSSW lemma 6}.  Part {\it(ii)} follows from the argument in the proof of Lemma~\ref{PSSW lemma 6}.  Part {\it(iii)} follows from part {\it (ii)} and part {\it (i)}.
\end{proof}

\begin{versionb}
\begin{proof}[Theorem~\ref{main theorem 3} part {\it (i)}]
Fix $p>2$.  As pointed out in the proof of Lemma~\ref{PSSW lemma 5}, we can find $T$ sufficiently large so that
\begin{equation*}
|R^{T}-(2n+2)n | \leq C_{1}\displaystyle\prod_{j=1}^{N} \| \nabla u (t-a_{j}) \|_{L^{2}}^{\delta_{j}} \cdot \displaystyle\prod_{k=1}^{N}\| \nabla u (t-b_{k}) \|_{C^{0}}^{\epsilon_{k}}
\end{equation*}
where $\delta := \sum_{j} \delta_{j}$ and $\epsilon := \sum_{k} \epsilon_{k}$ satisfy $\delta + \epsilon =1$, and $1>\delta \geq 2/p$.  By the uniform bound from Theorem~\ref{Pman thm}, we obtain that there exists a $T$ such that for all $t\geq T$, we have
\begin{equation*}
|R^{T}-(2n+2)n | \leq C_{2}\displaystyle\prod_{j=1}^{N} \| \nabla u (t-a_{j}) \|_{L^{2}}^{\tilde{\delta}_{j}}
\end{equation*}
where $\sum \tilde{\delta}_{j} = 2/p$.  Since $\int_{0}^{\infty} Y(t) dt < \infty$, H\"older's inequality yields
\begin{equation*}
\begin{aligned}
\int_{T}^{\infty} \|R^{T}-(2n+2)n\|^{p}_{C^{0}} dt &\leq C_{2} \int_{T}^{\infty}  \displaystyle\prod_{j=1}^{N}Y(t-a_{j})^{p\tilde{\delta}_{j}/2} dt \\
&\leq C_{2}  \displaystyle\prod_{j=1}^{N} \left( \int_{T}^{\infty} Y(t-a_{j}) dt\right)^{p\tilde{\delta}_{j}/2} < \infty,
\end{aligned}
\end{equation*}
proving the theorem.
\end{proof}
\end{versionb}
\begin{versionb}
\section{Concluding Remarks}

In the K\"ahler-Ricci flow, the results of Phong-Sturm \cite{PS}, and Phong-Song-Sturm-Weinkove \cite{PSSW} have had a great deal of success when the bisectional curvature is positive;  see for example \cite{PSSW1}. The Sasaki-Ricci flow in the presence of positive transverse holomorphic bisectional curvature was studied in \cite{WHe}, where it was pointed out that positivity is preserved along the flow.  As a result, under the assumption of positive holomorphic sectional curvature one can easily obtain the uniform boundedness in $C^{\infty}$ of the Riemann tensor by applying Theorem~\ref{Pman thm} and Theorem~\ref{BBS Thm}.  By Theorem~\ref{main theorem}, establishing convergence of the Sasaki-Ricci flow reduces to proving that conditions (A) and (C) hold.  In fact, in the presence of positive bisectional curvature, it was show in \cite{PSSW1} that the boundedness below of the Mabuchi K-energy is sufficient to guarantee the convergence of the K\"ahler-Ricci flow.  The key step in the proof is to use positive bisectional curvature to prove that condition (S) holds (cf. Theorem 1.2).  The results of the current paper suggest that similar conclusions hold in the Sasaki setting.  It was also shown in \cite{PSSW1} that condition (B) combined with the vanishing of the Futaki invariant is sufficient to obtain convergence of the K\"ahler-Ricci flow.  Moreover, it was shown by Cao and Zhu that condition (S) holds when the bisectional curvature is positive \cite{CaoZhu}; it would be interesting to know whether these results carry over to the Sasaki setting.  That is, can one adapt the techniques of \cite{CaoZhu} to show that positive transverse holomorphic bisectional curvature implies condition (T) holds?  Another potential application of our results is to study the Sasaki threefolds with transverse Riemann curvature which is `2-nonnegative' (cf. \cite{PS, PS1}). The obstructions to generalizing the results of \cite{PS} on K\"ahler manifolds with 2-nonnegative curvature tensor is to find an appropriate extension of 2-nonnegativity to the Sasaki setting, and to show it is preserved along the Sasaki-Ricci flow.
\end{versionb}

\begin{versiona}
\section{Appendix}

The main objective of the material in the appendix is to prove Proposition~\ref{Mab prop}.  Most of the work in this direction was done by Nitta and Sekiya in \cite{NS} where they generalized the bulk of the paper \cite{BM} to the Sasaki setting.  We begin by fixing a Sasaki structure $(S, \xi, \eta, \Phi, g)$, with $(2n+2)[\frac{1}{2}d\eta] \in c_{1}^{B}(S)$.  We will denote by $\mathcal{G}$ the identity component of the automorphism group of the Sasaki structure $(S,\xi, \Phi)$ ??? As before, we let $\mathcal{H} :=\mathcal{H}_{\eta}$ be the space of Sasaki metrics, identified by their Sasaki potentials, which are compatible with the structure $(\xi, \eta, \Phi, g)$, in the sense of Proposition~\ref{perturb prop}.  We will denote by $\mathfrak{E}$ the set of all Sasaki-Einstein metrics in $\mathcal{H}_{\eta}$.  Throughout this section, we will assume that $\mathfrak{E} \ne \emptyset$.  We will also denote by $Vol(S) := \int_{S} (\frac{1}{2}d\eta)^{2} \wedge \eta$ the volume of $S$.  We now define functionals $L_{\eta}, M_{\eta}, I_{\eta}$ and $J_{\eta}$ on $\mathcal{H}$ by
\begin{equation*}
\begin{aligned}
& L_{\eta}(\phi):=\frac{1}{V} \int_{a}^{b} dt \int_{S}\dot{\phi}_{t}(\frac{1}{2}d\eta_{\phi_{t}})^{n}\wedge \eta_{\phi_{t}},\\
&M_{\eta}(\phi) := -\frac{1}{V}\int_{a}^{b}dt \int_{S} \dot{\phi}_{t}(R^{T}(\phi_{t})-n(2n+2))d\eta_{\phi_{t}})^{n}\wedge \eta_{\phi_{t}},\\
&I_{\eta}(\phi) := \frac{1}{V} \int_{S}\phi\left((\frac{1}{2}d\eta)^{n}\wedge \eta - (\frac{1}{2}d\eta_{\phi})^{n} \wedge \eta_{\phi}\right),\\
&J_{\eta}(\phi):= \frac{1}{V} \int_{a}^{b} dt \int_{S}\int_{S}\dot{\phi}_{t}\left((\frac{1}{2}d\eta)^{n}\wedge \eta-(\frac{1}{2}d\eta_{\phi_{t}})^{n}\wedge \eta_{\phi_{t}}\right).
\end{aligned}
\end{equation*}
where $\phi_{t}:[a,b] \rightarrow \mathcal{H}$ is a smooth path with $\phi_{a} =0$, and $\phi_{b}=\phi$.  These functionals are the Sasaki version of the familiar functionals from K\"ahler geometry.  See NITTA-SEKIYA ref[2], PHONG-STURM overview, for more on these functionals and their roles in K\"ahler geometry, and NITTA-SEKIYA for their basic properties in the Sasaki setting.  

Let $h \in C^{\infty}_{B}$ denote the transverse Ricci potential of the metric $g$, normalized so that $\int_{S}(e^{h}-1)(\frac{1}{2}d\eta)^{n}\wedge \eta)=0$.  The following one-parameter families of Monge-Amp\`ere equations were the primary object of study in NITTA-SEKIYA;
\begin{equation}\label{NS 21}
\det (g^{T}_{\bar{j}i} + \dop{i}\dbop{j}\psi_{t}) = \det (g^{T}_{\bar{j}i}) \exp(-t(2n+2)\psi_{t} +h); \quad t\in[0,1]
\end{equation}
\begin{equation}\label{NS 22}
\det (g^{T}_{\bar{j}i} + \dop{i}\dbop{j}\phi_{t}) = \det (g^{T}_{\bar{j}i}) \exp(-t(2n+2)\phi_{t} -L_{\eta}(\phi_{t}) +h); \quad t\in[0,1]
\end{equation}
where $\psi_{t}, \phi_{t}$ are required to belong to $\mathcal{H}$.  When $t=1$, both of these equations are the transverse K\"ahler-Einstein equation.  As a consequence of El Kacimi-Alaoui's estimates, we have;
\begin{Theorem}[NITTA-S Theorem 4.2, Corollary 4.3]
If $t=0$, then equations~(\ref{NS 21}) and~(\ref{NS 22}) have solutions $\psi_{0}$ and $\phi_{0}$ respectively.  Moreover,  $\psi_{0}$ is unique up to an additive constant, and $\phi_{0}$ is unique and satisfies $L_{\eta}(\phi_{0}) =0$.
\end{Theorem}
NITTA AND SEKIYA obtained openness for the problem~(\ref{NS 22}) at times $\tau \in [0,1)$.

\begin{prop}[NITTA-S Proposition 4.4, Theorem 4.10]\label{NS prop 4.4}
Let $0 \leq \tau <1$.  Suppose that equation~(\ref{NS 22}) has a solution $\phi_{\tau}$ at $t=\tau$.  Then for some $\epsilon >0$, $\phi_{\tau}$ extends uniquely to a smooth one parameter family $\{ \phi_{t} : t\in[0,1]\cap [\tau-\epsilon, \tau +\epsilon]\}$ of solutions to~(\ref{NS 22}).
\end{prop}

In order to prove that the extension is unique, one makes the key observation, following Bando and Mabuchi,  that the Mabuchi K-energy is decreasing along the method of continuity.  This observation is crucial for the proof of Proposition~\ref{Mab prop}.

\begin{Lemma}[NITTA-S Lemma 4.9]\label{NS Lemma 4.9}
Let $\{ \phi_{t} : t\in [0,1]\}$ be an arbitrary smooth family of solutions of~(\ref{NS 22}).  Then,
\begin{equation*}
\frac{d}{dt}M_{\eta}(\phi_{t}) = -(1-t)(2n+2) \frac{d}{dt} \left(I_{\eta}(\phi_{t}) - J_{\eta}(\phi_{t})\right) \leq 0
\end{equation*}
\end{Lemma}

The main difficulty which arises in BANDO-MABUCHI, NITTA-SEKIYA is obtaining closedness at $\tau =1$.  Indeed, it is not  true in general that closedness holds at $\tau =1$, unless the IDENTITY COMPONENT O THE AUTOMORPHISM GROUP?? is trivial.  It is possible, however, to obtain backwards openness at $\tau =1$, by employing the bifurcation technique of Bando and Mabuchi REF.  Let $\mathcal{O}$ be an arbitrary $\mathcal{G}$ orbit in $\mathfrak{E}$.  Consider the positive function $\iota := (I_{\eta} -J_{\eta}) \big|_{\mathcal{O}}: \mathcal{O} \rightarrow \mathbb{R}$.  The functional $\iota$ is related to backwards openness at $\tau =1$.

\begin{Lemma}[NITTA-SEKIYA Lemma 4.3]\label{min lemma}
The functional $\iota$ is proper.  In particular, its minimum is always attained at some point in the orbit $\mathcal{O}$.
\end{Lemma}
 
\begin{prop}[NITTA-SEKIYA Proposition 4.15]\label{openness lemma}
For every critical point $g_{SE} \in \mathcal{O}$ of $\iota$ with non-degenerate Hessian, $\phi_{1} := \phi(g_{SE})$, the transverse K\"ahler potential of $g_{SE}$, can be extended to a smooth family $\{ \phi_{t} : t \in [1-\epsilon, 1]\}$ of solutions to~(\ref{NS 21}) for some $\epsilon >0$.
\end{prop}

Bando and Mabuchi observed that it was possible to ensure that at the minimum point guaranteed by Lemma~\ref{min lemma}, the functional $\iota$ could be made to have non-degenerate Hessian, by perturbing the initial condition slightly.  Let $g_{SE}$ be the minimum of $\iota$ on the orbit $\mathcal{O}$, and let $\phi_{1} = \phi(SE)$ be the transverse K\"ahler potential of $g_{SE}$. Fix $\delta \in (0,1)$, and define $g^{\delta}$ to be the Sasaki metric defined by $\eta + \delta d_{B}^{c}\phi_{1}$.  In particular, the transverse metric is given by
\begin{equation*}
(g^{\delta})^{T} := g^{T} + \delta \dop{B}\dbar_{B} \phi_{1} = (1-\delta)g^{T} + \delta g^{T}_{SE}.
\end{equation*}
Let $\iota^{\delta}$ denote the function $\iota$ computed with respect to the metric $g^{\delta}$.  For every $\delta >0$, one checks that $\iota^{\delta}$ has positive definite hessian at $g_{SE}$ (cf. Remark 4.17 NITTA-SEKIYA).  Thus, by Proposition~\ref{openness lemma} we have backwards openness, and by Proposition~\ref{NS prop 4.4}, this extends to a one parameter family $\phi^{\delta}_{t}$ for $t \in [0,1]$.  It follows from Lemma~\ref{NS Lemma 4.9} that 
\begin{equation*}
M\left(g^{T}_{SE}, (g^{\delta})^{T} + \dop{B}\dbar_{B}\phi_{0}^{\delta}\right) \geq 0.
\end{equation*}
By definition $\phi_{0}^{\delta}$, solves
\begin{equation*}
\det ((g^{\delta})^{T}_{\bar{j}i} + \dop{i}\dbop{j}\phi^{\delta}_{0}) = \det ((g^{\delta})^{T}_{\bar{j}i})  e^{h^{\delta}}
\end{equation*}
where $h^{\delta}$ is the transverse Ricci potential of $g^{\delta}$.  As $\delta \rightarrow 0$, $g^{\delta} \rightarrow g$ in $C^{0,\alpha}$, and so the estimates of ElKac imply that Ric.

\end{versiona}

\end{document}